\documentclass[12pt]{amsart}

\usepackage[margin=1.15in]{geometry}
\usepackage[T1]{fontenc}
\usepackage[utf8]{inputenc}
\usepackage{microtype}
\usepackage{amsmath,amssymb,amsthm,mathtools}
\usepackage[dvipsnames]{xcolor}
\usepackage[colorlinks=true,linkcolor=blue,citecolor=blue,urlcolor=blue]{hyperref}
\usepackage{tikz-cd} 
\usepackage{lmodern}
\newtheorem{theorem}{Theorem}[section]
\newtheorem{proposition}[theorem]{Proposition} 
\newtheorem{corollary}[theorem]{Corollary} 
\newtheorem{lemma}[theorem]{Lemma}
 
\theoremstyle{definition}

\theoremstyle{remark}

\newtheorem{example}[theorem]{Example}

\newcommand{\K}{\mathbb K} 
\newcommand{\F}{\mathbb F}

\newcommand{\C}{\mathbb C}

\newcommand{\fg}{\mathfrak g}
\newcommand{\fh}{\mathfrak h}
\newcommand{\fa}{\mathfrak a}

\newcommand{\fq}{\mathfrak q}

\newcommand{\fp}{\mathfrak p}
\newcommand{\fm}{\mathfrak m}
\newcommand{\fk}{\mathfrak k}
\newcommand{\fs}{\mathfrak s}
\newcommand{\End}{\operatorname{End}}
\newcommand{\Der}{\operatorname{Der}}
\newcommand{\Lie}{\operatorname{Lie}}
\newcommand{\Span}{\operatorname{span}}
\newcommand{\ad}{\operatorname{ad}}
\newcommand{\Spec}{\operatorname{Spec}}
\newcommand{\Ass}{\operatorname{Ass}}
\newcommand{\Supp}{\operatorname{Supp}}
\newcommand{\Frac}{\operatorname{Frac}}
\newcommand{\LND}{\operatorname{LND}}
\newcommand{\LFD}{\operatorname{LFD}}
\newcommand{\trdeg}{\operatorname{trdeg}}
\newcommand{\Hom}{\operatorname{Hom}}

\newcommand{\Aut}{\operatorname{Aut}}

\newcommand{\MinSupp}{\operatorname{MinSupp}}
\theoremstyle{plain}

\newtheorem{introtheorem}{Theorem}

\newtheorem{introcorollary}{Corollary}

\newtheorem{introcorollaryb}{Corollary}

\newtheorem{introcor}[introtheorem]{Corollary}
\newcommand{\xc}[1]{\vspace{.2cm}

\noindent {\em #1} }

\title[Rigidity for locally finite derivations]{Rigidity for Lie algebras of locally finite derivations}

\author{Mohamed Ali Belabbas}
\address{Coordinated Science Laboratory,  Department of Electrical and Computer   Engineering and Department of Mathematics, University of Illinois at Urbana-Champaign}
\email{belabbas@illinois.edu}
\date{}


\begin{document}

\begin{abstract}

Let $A$ be a finitely generated commutative algebra over a field $\mathbb K$ of characteristic zero. We prove that every finitely generated Lie subalgebra of $\operatorname{Der}_{\mathbb K}(A)$ whose elements are locally finite on $A$ is finite-dimensional. Consequently, for a Lie subalgebra generated by finitely many locally finite derivations, the following are equivalent: it is finite-dimensional, it acts locally finitely on $A$, and all its elements are locally finite. A key ingredient is a second theorem: every Lie subalgebra of $\operatorname{Der}_{\mathbb K}(A)$ whose elements are locally nilpotent is solvable; when $A$ is reduced, its derived length is at most $\dim A$, and this bound is sharp. In general, the $(\dim A)$-th term of its derived series is nilpotent, of class bounded in terms of $A$ only. We also prove that every solvable Lie subalgebra of $\operatorname{Der}_{\mathbb K}(A)$ generated by finitely many locally finite derivations is finite-dimensional and acts locally finitely on $A$. For an affine variety $X$ over an algebraically closed field, we deduce that a subgroup of $\operatorname{Aut}(X)$ generated by finitely many connected algebraic subgroups is algebraic if and only if every element of the Lie algebra generated by their tangent algebras is locally finite. For two unipotent one-parameter subgroups, this gives an answer to a 2005 problem of Popov. Our results also characterize polynomial control systems admitting an exact finite-dimensional bilinear realization by polynomial observables containing the state coordinates. 
The proofs rest on the introduction of a cofinite ideal meeting the closure of every associated point of $\operatorname{Spec} A$: on the subalgebra of derivations vanishing to second order along the corresponding finite subscheme, local finiteness forces local nilpotence. The intersection of that subalgebra with $L$ is therefore solvable by the second theorem, and has finite codimension in $L$; together with local finiteness of the adjoint action, this yields finite-dimensionality.

\end{abstract}

\maketitle

\section{Introduction}\label{sec:intro}

Let $A$ be a finitely generated commutative algebra over a field of characteristic zero. Every element of $\Der_\K(A)$ is a regular vector field on $\Spec A$.  A derivation $D \in \Der_\K(A)$ is {\em locally finite on $A$}\footnote{The terms \emph{locally finite} and \emph{pointwise finite} (likewise \emph{locally nilpotent} and \emph{pointwise nilpotent}) both refer to the condition holding at each $a\in A$; for derivations the term ``locally'' has become standard, and we follow it.} if, for every $a\in A$, the space
$$
\Span_\K\{a,Da,D^2a,\ldots\}
$$
is finite-dimensional. It is {\em locally nilpotent on $A$} if for every $a\in A$ there is $n\geq1$ with $D^n a=0$. We put
\begin{align*}
\LFD(A)&:=\{D\in\Der_\K(A):D\text{ is locally finite on }A\} \mbox{ and }\\
\LND(A)&:=\{D\in\Der_\K(A):D\text{ is locally nilpotent on }A\}.
\end{align*}
We call  a Lie subalgebra $L\subseteq\Der_\K(A)$ {\em locally finite on $A$}, or simply a {\em locally finite subalgebra} when the algebra is clear from context, if $A$ is a locally finite $L$-module: every $a\in A$ is contained in some finite-dimensional subspace $U\subseteq A$ satisfying $D(U)\subseteq U$ for all $D\in L$. Thus {\em elementwise local finiteness} (i.e., inclusion in $\LFD(A)$) allows $U$ to depend on $D$ and $a$, whereas local finiteness of the action allows $U$ to depend only on $a$.  Our main result shows that every finitely generated Lie subalgebra $L\subseteq\LFD(A)$ is finite-dimensional. This rigidity result has several applications. Together with the other results established here, it answers questions of Popov~\cite{PopovProblems2005}, Kraft and Zaidenberg~\cite{KraftZaidenberg2024}, Arzhantsev and Zaidenberg~\cite{ArzhantsevZaidenberg2025}, and Chitayat, Daigle and Regeta~\cite{ChitayatDaigleRegeta2026}, which we discuss below. Our interest in this question was independently motivated by a seemingly unrelated problem arising in applied mathematics and engineering, namely in the super-linearization of control systems; we describe the algebraic side first, and return to the control-theoretic motivation at the end of this section.

Let $X$ be an affine variety over an algebraically closed field of characteristic zero. One of the ways to study its automorphism group $\Aut(X)$ is through its algebraic subgroups. A natural question is then whether the subgroup generated by finitely many algebraic subgroups is again algebraic. For two unipotent one-parameter subgroups of $\Aut(\mathbb A^n)$, a form of this question was posed by Popov~\cite[Problem~3.1]{PopovProblems2005}. 


Locally finite derivations are the infinitesimal counterpart of algebraic group actions. If an affine algebraic group $G$ acts regularly on an affine variety $X$, then $\K[X]$ is a locally finite $G$-module, i.e., every $f\in\K[X]$ is contained in a finite-dimensional $G$-stable subspace, and every element of $\Lie(G)$ acts as a locally finite derivation~\cite{BialynickiBirula1973, Freudenburg2006, springer1998linear}. Conversely, over an algebraically closed field of characteristic zero, the one-dimensional Lie algebra spanned by a locally finite derivation is contained in the Lie algebra of an algebraic subgroup of $\Aut(X)$; this is the one-dimensional case of Cohen and Draisma's integration theorem~\cite[Theorem~1]{CohenDraisma2003}.

Kraft and Zaidenberg have proved that a subgroup of $\Aut(X)$ generated by connected algebraic subgroups is algebraic exactly when the Lie algebra generated by the tangent algebras of these subgroups is finite-dimensional~\cite[Theorem~A]{KraftZaidenberg2024}. In view of this, the problem is whether finitely many such infinitesimal actions assemble into the action of a single algebraic group. In general, it is not the case: on $\K[x,y]$, the locally nilpotent derivations $x^2\partial_y$ and $y^2\partial_x$ generate an infinite-dimensional Lie algebra, and their Lie bracket is not locally finite. Indeed, $[x^2\partial_y,y^2\partial_x]=2x^2y\,\partial_x-2xy^2\,\partial_y$ maps $x^ay^b$ to $2(a-b)x^{a+1}y^{b+1}$; its iterates applied to $x$ are the monomials $2^mx^{m+1}y^m$, $m\geq0$, so the orbit of $x$ spans an infinite-dimensional subspace of $\K[x,y]$.

This leads to two questions. First, if the Lie algebra $L:=\Lie_\K(D_1,\ldots,D_m)$ generated by locally finite derivations $D_1,\ldots,D_m$ is assumed {\em finite-dimensional}, can the individual finite-dimensional orbits fail to assemble into {\em common} finite-dimensional $L$-invariant subspaces? This is the question raised by Kraft and Zaidenberg~\cite[Questions~2 and~3]{KraftZaidenberg2024}: if $L$ is finite-dimensional and generated by locally finite derivations, must $L$ {\em act} locally finitely on $A$?  It is the case: over $\C$, this follows from a lemma of Kac~\cite[\S1.2, Lemma~(b)]{Kac1985}, which appeared in a different context, namely the construction of groups from infinite-dimensional Lie algebras. Section~\ref{sec:forward} proves this over an arbitrary field of characteristic zero.

Second, can $L$ be infinite-dimensional while every element of $L$,  not only its generators, is locally finite? In the example above the two obstructions arrive together: $L$ is infinite-dimensional \emph{and} contains a non-locally finite element, namely the bracket of the generators. The question is whether they must. This question naturally leads to first considering the adjoint action of the derivations on $\Der_\K(A)$. We  show that if $D_i$ is locally finite on $A$, then $\ad_{D_i}$ is locally finite on $\Der_\K(A)$ by Lemma~\ref{lem:adjoint-lf}. However, this alone turns out to be insufficient to conclude that $L$ is finite-dimensional: Example~\ref{ex:not-derivations} exhibits two locally finite operators on a vector space whose generated Lie algebra acts locally finitely and has generators with nilpotent adjoint action, and is {\em nevertheless infinite-dimensional}. The adjoint action is one of the ingredients of the proof, but its local finiteness alone is insufficient; one must also use that the elements of $L$ are derivations of a finitely generated commutative algebra. Theorem~\ref{thm:main} answers the question in the negative, and it is the geometry of $\Spec A$ that allows one to settle it: it furnishes a finite closed subscheme along which a subalgebra of finite codimension in $L$ vanishes to second order, and on such derivations local finiteness collapses to local nilpotence. The two failure modes, namely $L$ being infinite-dimensional and $L$ containing an element that is not locally finite, therefore always occur together: they are the same obstruction. Equivalently, the only obstruction to finitely many connected algebraic subgroups generating an algebraic subgroup is an element of the generated Lie algebra that is not locally finite (Corollary~\ref{cor:algebraicity-generated-subgroups}).

\subsection*{Main results}

\begin{introtheorem}[Rigidity for Lie algebras of locally finite derivations]\label{thm:main}
Let $\K$ be a field of characteristic zero, and let $A$ be a finitely generated commutative $\K$-algebra. Every finitely generated Lie subalgebra $L\subseteq\LFD(A)$ is finite-dimensional.
\end{introtheorem}

Combined with results of Kac and Fernando, Theorem~\ref{thm:main} yields the equivalence used throughout the paper:

\begin{introcorollary}\label{cor:equiv}
Let $\K$ and $A$ be as in Theorem~\ref{thm:main}, and let $L\subseteq\Der_\K(A)$ be a Lie subalgebra generated by finitely many locally finite derivations. Then the following conditions are equivalent:
\begin{enumerate}
\item $\dim_\K L<\infty$;
\item $L$ acts locally finitely on $A$;
\item $L\subseteq\LFD(A)$.
\end{enumerate}
\end{introcorollary}
The implication $(1)\Rightarrow(2)$ is representation-theoretic and holds for an arbitrary commutative $\K$-algebra (see Proposition~\ref{prop:fd-direction}, the extension of Kac's lemma to an arbitrary field of characteristic zero); and $(2)\Rightarrow(3)$ is immediate.

Finite generation is essential: for $A:=\K[x,y]$, the Lie algebra $L:=\K[y]\partial_x$ is infinite-dimensional and contained in $\LND(A)$, while the $L$-submodule generated by $x$ contains $\K[y]$. Theorem~\ref{thm:main} gives $(3)\Rightarrow(1)$, and is where the geometry of $A$ enters, and where the main work of the paper lies.

Applied to the finitely generated subalgebras of an arbitrary Lie algebra of locally finite derivations, Theorem~\ref{thm:main} nonetheless yields a conclusion in which no finite-generation hypothesis is imposed on the Lie algebra.

\begin{introcorollaryb}\label{cor:global-lf}
Let $\K$ and $A$ be as in Theorem~\ref{thm:main} and let $\fh\subseteq\LFD(A)$ be a Lie subalgebra. Then every finitely generated Lie subalgebra of $\fh$ is finite-dimensional and acts locally finitely on $A$. Consequently $\fh$ is locally finite as an abstract Lie algebra. Furthermore,  for every finite subset $S\subseteq\fh$ and every $a\in A$ there is a finite-dimensional subspace $U\subseteq A$ with $a\in U$ and $D(U)\subseteq U$ for all $D\in S$.
\end{introcorollaryb}
\noindent The last assertion is termed {\em weak local finiteness} by Chitayat, Daigle and Regeta, who establish it under a solvability hypothesis and ask whether elementwise local finiteness implies it in general~\cite[Section~1]{ChitayatDaigleRegeta2026}. Corollary~\ref{cor:global-lf} answers their question affirmatively for finitely generated commutative algebras in characteristic zero, with no solvability hypothesis.

Our second main result concerns the structure of the Lie subalgebras of $\Der_\K(A)$ contained in $\LND(A)$; it also supplies the solvability step in the proof of Theorem~\ref{thm:main}.

\begin{introtheorem}[Solvability of locally nilpotent derivations]\label{thm:lnd-solvable}
Let $\K$ be a field of characteristic zero and let $A$ be a nonzero finitely generated commutative $\K$-algebra. Every Lie subalgebra $\fa\subseteq\LND(A)$ is solvable; more precisely, $\fa^{(\dim A)}$ is nilpotent, with nilpotency class bounded in terms of $A$ only. If $A$ is reduced, then $\fa^{(\dim A)}=0$, that is, $\operatorname{dl}(\fa)\leq\dim A$.
\end{introtheorem}

There is no finite-generation hypothesis on $\fa$, and its solvability cannot be upgraded to nilpotence: $\K\partial_x\oplus\K[x]\partial_y\subseteq\LND(\K[x,y])$, while its lower central series stabilizes at $\K[x]\partial_y$. The dimension bound requires $A$ to be reduced: $L=\K\partial_x\oplus\K[x]\epsilon\partial_x$ on $\K[x,\epsilon]/(\epsilon^2)$ has derived length two, but $\dim A=1$. Here $L^{(1)}=\K[x]\epsilon\partial_x$ is abelian, as indicated by Theorem~\ref{thm:lnd-solvable}, and $L$ itself is not nilpotent, since $[\partial_x,f\epsilon\partial_x]=f'\epsilon\partial_x$ for $f\in\K[x]$; thus the index $\dim A$ in $\fa^{(\dim A)}$ cannot be lowered. The example is easily generalized to every dimension $d$, using the triangular Lie algebra of Lemma~\ref{lem:triangular-dl} in place of $\K\partial_x$, and $\epsilon\K[x_1,\ldots,x_d]\partial_{x_d}$ in place of $\K[x]\epsilon\partial_x$. The bound $\operatorname{dl}(\fa)\leq\dim A$ for reduced algebras is sharp in every dimension (Lemma~\ref{lem:triangular-dl}); for domains, Proposition~\ref{prop:domain-lnd-bound} shows that $\operatorname{dl}(\fa)\leq\trdeg_F\Frac(A)$, where $F$ is the fraction field of the common kernel $\bigcap_{D\in\fa}\ker D$. 

The nilpotence of $\fa^{(\dim A)}$ follows from a stronger statement, that we prove in Section~\ref{sec:LNDAA} (Proposition~\ref{prop:uniform-kernel}): the ideal $\fk\subseteq\fa$ of derivations acting trivially on $A/\sqrt{0}$ contains $\fa^{(\dim A)}$, and all  sufficiently long products of elements of $\fk$ vanish on $A$, with a bound depending {\em only} on $A$. This bound cannot be chosen in terms of $\dim A$ and of the nilpotency index of $\sqrt{0}$ alone. For example, let $A_m:=\K[\epsilon_1,\ldots,\epsilon_m]/(\epsilon_i\epsilon_j:1\leq i,j\leq m)$. Then $\dim A_m=0$, and the nilradical $\K\epsilon_1\oplus\cdots\oplus\K\epsilon_m$ of $A_m$ has square zero for every $m$; yet the strictly upper triangular matrices, acting on this nilradical, give a Lie algebra of locally nilpotent derivations of $A_m$ with nilpotency class $m-1$.

\subsection*{Consequences of the main results}

Corollary~\ref{cor:equiv} yields the algebraicity criterion announced above:

\begin{introcor}[Algebraicity of groups generated by algebraic subgroups]\label{cor:algebraicity-generated-subgroups}
Assume that $\K$ is algebraically closed of characteristic zero. Let $X$ be an affine $\K$-variety, put $A:=\K[X]$, and let $G_1,\ldots,G_m\subseteq\Aut(X)$ be connected algebraic subgroups. Let $G:=\langle G_1,\ldots,G_m\rangle\subseteq\Aut(X)$ be the subgroup they generate, and set
$$
L:=\Lie_\K\bigl(\Lie(G_1)\cup\cdots\cup\Lie(G_m)\bigr)\subseteq\Der_\K(A).
$$
Then the following conditions are equivalent:
\begin{enumerate}
\item $G$ is an algebraic subgroup of $\Aut(X)$;
\item $\dim_\K L<\infty$;
\item $L$ acts locally finitely on $A$;
\item $L\subseteq\LFD(A)$.
\end{enumerate}
When these conditions hold, $\Lie(G)=L$.
\end{introcor}

\noindent The equivalence of (1) and (2), and the equality $\Lie(G)=L$, are due to Kraft and Zaidenberg~\cite[Theorem~A]{KraftZaidenberg2024}; the equivalence of (2) with (3) and (4) is Corollary~\ref{cor:equiv}.

Popov asked when the minimal closed subgroup of $\Aut(\mathbb A^n)$ containing two prescribed one-parameter unipotent subgroups is finite-dimensional~\cite[Problem~3.1]{PopovProblems2005}. Corollary~\ref{cor:popov} provides an answer to this question: if $D,E\in\LND(\K[x_1,\ldots,x_n])$ are the infinitesimal generators of these subgroups, we obtain the intrinsic criterion
$$
\dim G<\infty\quad\Longleftrightarrow\quad\Lie_\K(D,E)\subseteq\LFD(\K[x_1,\ldots,x_n]).
$$

The local finiteness of the adjoint action established in Section~\ref{sec:adjoint} also yields the following result, which settles Question~1 of Arzhantsev and Zaidenberg~\cite{ArzhantsevZaidenberg2025}.

\begin{introtheorem}[Solvable Lie algebras generated by locally finite subalgebras]\label{thm:mainsolvable} Let $\K$ be a field of characteristic zero, let $A$ be a finitely generated commutative $\K$-algebra, and let $\fh$ be a solvable Lie algebra generated by finitely many locally finite Lie subalgebras $\fh_i\subseteq\Der_\K(A)$. Then $\fh$ is finite-dimensional and acts locally finitely on $A$.
\end{introtheorem}

\subsection*{Outline of the proof of Theorem~\ref{thm:main}}

The tool that reduces Theorem~\ref{thm:main} to Theorem~\ref{thm:lnd-solvable} is the introduction of a {\em cofinite separating ideal}, namely an ideal $I\subseteq A$ such that $A/I$ is finite-dimensional and $\bigcap_{q\geq1}I^q=0$. Such ideals always exist under our assumptions: select a closed point in the closure of every associated point of $\Spec A$, including the embedded associated points (see Lemma~\ref{lem:existencecosepideal}). Put
$$
\mathcal D_I^2:=\{D\in\Der_\K(A):D(A)\subseteq I^2\}.
$$
If $D\in\mathcal D_I^2$, then $D(I^q)\subseteq I^{q+1}$ and $D^n(a)\in I^{n+1}$ for every $a\in A$ and $q,n\geq1$. When $D$ is locally finite, the span $U$ of $a,Da,D^2a,\ldots$ is finite-dimensional, so the decreasing chain $U\cap I^q$ stabilizes, necessarily at $U\cap\bigcap_qI^q=0$; since $D^n(a)\in U\cap I^{n+1}$, it follows that $D$ is locally nilpotent. Thus
$$
\LFD(A)\cap\mathcal D_I^2=\LND(A)\cap\mathcal D_I^2.
$$
The subalgebra $H_I:=L\cap\mathcal D_I^2$ has finite codimension in $L$, and the proof of Theorem~\ref{thm:main} is summarized by
\begin{equation}\label{eq:roadmap}
\begin{gathered}
H_I\subseteq\LFD(A)\cap\mathcal D_I^2\ \xRightarrow{\ \text{Lemma~\ref{lem:second-order-lnd}}\ }\ H_I\subseteq\LND(A)\ \xRightarrow{\ \text{Theorem~\ref{thm:lnd-solvable}}\ }\ H_I\ \text{solvable}\\
\xRightarrow{\ \text{Lemma~\ref{lem:Petravchuk}}\ }\ \exists\,\text{a solvable ideal of finite codimension in }L\ \xRightarrow{\ \text{Proposition~\ref{prop:virtually-solvable}}\ }\ \dim_\K L<\infty.
\end{gathered}
\end{equation}
The last implication uses local finiteness of the adjoint action of the generators.

The argument in fact precisely locates the obstruction: finite-dimensionality of $L$ is detected  on $H_I$, and an infinite-dimensional $L$ contains derivations that are not locally finite and {\em vanish to arbitrarily high order} along the fixed finite subscheme $\Spec(A/I)$. 

Theorem~\ref{thm:intro-localized} in Section~\ref{sec:mainproof} makes this precise: with $H_I:=L\cap\mathcal D_I^2$ as above, finite-dimensionality of $L$  is equivalent to nilpotence,  solvability,  elementwise local finiteness, and elementwise local nilpotence of $H_I$.

Two consequences of the fact that $L$ is a Lie algebra of {\em derivations} will appear repeatedly in what follows. The first is that a derivation is determined by its values on a set of algebra generators, so that restriction to a finite-dimensional subspace containing such a set is faithful (Lemma~\ref{lem:faithful-restriction}). The second is that, through the Leibniz rule, a statement of the form $D(A)\subseteq I$ for an ideal $I\subseteq A$ becomes a statement about the {\em powers} of $I$. If $D(A)\subseteq I^2$, then $D(I^q)\subseteq I^{q+1}$ for every $q\geq1$, which is what underlies $\mathcal D_I^2$ above; and if $D$ maps $A$ into the nilradical $\mathcal N$, then $D$ acts $A/\mathcal N$-linearly on each quotient $\mathcal N^i/\mathcal N^{i+1}$, which is what makes Lemma~\ref{lem:module-engel} available in the proof of Theorem~\ref{thm:lnd-solvable}. Neither consequence is geometric. A separate geometric ingredient in the proof of Theorem~\ref{thm:main} is the existence of a cofinite separating ideal (Lemma~\ref{lem:existencecosepideal}), a statement about $A$ alone in which derivations play no part.

\subsection*{Related work}

The study of Lie algebras of locally nilpotent derivations is closely related to the classical problem of triangularization. Rentschler proved that every locally nilpotent derivation of $\K[x,y]$ is triangularizable~\cite{Rentschler1968}; Bass showed that triangularizability already fails on $\mathbb A^3$~\cite{Bass1984}; subsequent work of Freudenburg and Daigle developed triangularizability criteria and relative forms of Rentschler's theorem~\cite{Freudenburg1995,Daigle1996,DaigleFreudenburg1998}. For Lie algebras of locally nilpotent derivations, Bavula studied the structure of triangular Lie algebras of polynomial derivations~\cite{Bavula2013}. Petravchuk and Sysak proved nilpotence in the finite-dimensional domain case and simultaneous triangularization on the affine plane, and observed that the structure of such Lie algebras remains open beyond the finite-dimensional and two-variable settings~\cite{PetravchukSysak2017}. Skutin studied maximal Lie subalgebras among the locally nilpotent derivations of a polynomial ring~\cite{Skutin}; building on this work, Perepechko proved that such Lie algebras are solvable on irreducible affine varieties over an algebraically closed field~\cite[Corollary~3.4]{Perepechko2023}. Makedonskyi and Petravchuk studied nilpotent and solvable Lie algebras of derivations and obtained bounds on their derived length in terms of rank~\cite{MakedonskyiPetravchuk2014}. More recently, Bezushchak, Petravchuk and Zelmanov established local nilpotence of arbitrary Lie subalgebras contained in $\LND(A)$~\cite[Theorem~1.4]{BezushchakPetravchukZelmanov2024}.

Local finiteness has been studied both for individual derivations and in connection with Lie algebras and algebraic group actions. Bass and Meisters, and later van den Essen, studied locally finite derivations and their applications to polynomial flows~\cite{BassMeisters1985,VanDenEssen1992,VanDenEssen1994}. Cohen and Draisma studied the integration problem for Lie algebras of vector fields~\cite{CohenDraisma2003}, while Kraft and Zaidenberg study algebraically generated subgroups of automorphism groups and their Lie algebras~\cite{KraftZaidenberg2024}. Chitayat, Daigle and Regeta develop a systematic study of locally finite sets of derivations; in the solvable quasi-affine setting they obtain, among other structural results, a countable increasing filtration by locally finite Lie subalgebras~\cite{ChitayatDaigleRegeta2026}. Arzhantsev and Zaidenberg study Borel subalgebras of Lie algebras of vector fields and formulate finiteness questions for solvable Lie algebras generated by locally finite subalgebras~\cite{ArzhantsevZaidenberg2025}; Zaidenberg proves the affine-plane case~\cite{Zaidenberg2026}.

\subsection*{A control-theoretic motivation} The rigidity theorem also settles a problem from control theory, which is in fact how we first arrived at it. We briefly describe the control-theoretic content here.  Linearization is a  recurring theme in dynamics: near a hyperbolic equilibrium, the Hartman--Grobman theorem~\cite{Grobman1959,Hartman1960} straightens a flow into its linear part,  locally and up to homeomorphism. In  applications, one often simply uses a  first-order linear approximation, again valid only locally. Super-linearization~\cite{BelabbasChen2023} trades dimension for exactness: by embedding the state space into a higher-dimensional one, some nonlinear systems become globally and exactly linear. The linearizability of flows by embeddings has recently been studied from a geometric standpoint by Kvalheim and Arathoon~\cite{KvalheimArathoon2026}. 

Consider now a control-affine polynomial system
\begin{equation}\label{eq:mainsys}
\frac{d}{dt}x(t)=f_0(x(t))+u_1(t)f_1(x(t))+\cdots+u_r(t)f_r(x(t)),\qquad x(t)\in\K^n,\quad\K\in\{\mathbb R,\mathbb C\}.
\end{equation}
Each vector field $f_i$ induces a derivation $D_i$ of the polynomial functions $\K[x_1,\ldots,x_n]$ (called {\em observables} in that literature). A \emph{super-linearization}, which can also be seen as an exact finite-dimensional bilinear realization or Koopman lift~\cite{Koopman1931,BruntonKoopman2022}, is a finite-dimensional subspace $U\subseteq\K[x_1,\ldots,x_n]$ containing $1$ and the coordinate functions and invariant under all the $D_i$. In a basis of $U$, system~\eqref{eq:mainsys} lifts exactly to
$$
\dot z=A_0z+u_1A_1z+\cdots+u_rA_rz,
$$
with constant matrices $A_i$, and the original state is recovered by projection. Such lifts are practically useful, since they make it possible to bring the machinery of linear and bilinear control to bear on nonlinear systems. For one vector field such a realization is equivalent to local finiteness of its derivation. The family~\eqref{eq:mainsys} is in effect a parametric family of vector fields. For such a family, a \emph{simultaneous} linearization in the sense of Hartman--Grobman, that is, a single change of variables straightening all of $f_0,\ldots,f_r$ at once, cannot in general exist. For simultaneous super-linearization, by contrast, the rigidity theorem provides a complete answer: Corollary~\ref{cor:superlinearization} shows that system~\eqref{eq:mainsys} admits an exact finite-dimensional bilinear realization by polynomial observables containing the state coordinates if and only if $\Lie_\K(D_0,\ldots,D_r)\subseteq\LFD(\K[x_1,\ldots,x_n])$.

\subsection*{Organization of the paper}
Section~\ref{sec:forward} establishes the forward direction of Corollary~\ref{cor:equiv}, extending Kac's lemma to an arbitrary base field, and records the faithful restriction lemma used repeatedly in the sequel. Section~\ref{sec:adjoint} establishes local finiteness of the adjoint action, a finite-dimensionality criterion for Lie algebras with a solvable ideal of finite codimension (Proposition~\ref{prop:virtually-solvable}), and Theorem~\ref{thm:mainsolvable}. Section~\ref{sec:vanishing-order} constructs cofinite separating ideals and studies the filtration $\mathcal D_I^\bullet$. Section~\ref{sec:mainproof} then proves the rigidity theorems, using Theorem~\ref{thm:lnd-solvable} as an input. Section~\ref{sec:LNDAA} proves that theorem, first for reduced algebras and then for nilpotent thickenings. Section~\ref{sec:consequences} develops the automorphism-group and control-theoretic applications.

\subsection*{Conventions}

Throughout, $\K$ denotes a field of characteristic zero and $V$ is a $\K$-vector space. All algebras are commutative, associative, and unital; an {\em affine algebra} is a finitely generated commutative $\K$-algebra. All Lie algebras and their representations are over $\K$. For an algebra $A$, we write $\Der_\K(A)$ for the Lie algebra of $\K$-derivations, and $\mathcal N:=\sqrt{0}$ for the nilradical. The notation $\Lie_\K(S)$ denotes the Lie subalgebra generated by $S$, and $\ad_x(y):=[x,y]$. We take the derived length of the zero Lie algebra to be zero and that of a nonzero abelian Lie algebra to be one. We write $\gamma_1(\fg):=\fg$ and $\gamma_{j+1}(\fg):=[\fg,\gamma_j(\fg)]$ for the lower central series. 

\section{From finite-dimensionality to local finiteness of the action}\label{sec:forward}

We recall the Fernando--Kac theorem, which states that the elements of a finite-dimensional Lie algebra $\fa$ that act locally finitely on a module $M$ form a {\em subalgebra} of $\fa$:
\begin{proposition}[Fernando--Kac {\cite[\S1.2, Lemma~(b), pp.~170--171]{Kac1985}; \cite[Corollary~2.7, p.~762]{Fernando1990}}]\label{prop:FK}
Let $\fa$ be a finite-dimensional Lie algebra over an algebraically closed field $\overline\K$ of characteristic zero, let $M$ be an $\fa$-module that is finitely generated as a $U(\fa)$-module and let $\rho:\fa\to\End_{\overline\K}(M)$ be the module action. Then the set
$$
  \fh:=\{x\in\fa:\rho(x)\text{ is locally finite on }M\}
$$
is a Lie subalgebra of $\fa$.
\end{proposition}

Over $\mathbb C$, the next proposition is Kac's Lemma~1.2(b)~\cite[\S1.2, Lemma~(b), pp.~170--171]{Kac1985}; when applied to the subalgebra of $\fa$ generated by the locally finite elements, it also gives Proposition~\ref{prop:FK} over $\C$. Kac's proof uses in an essential way the fact that $\C$ is the base field: conjugation by $\exp(t\ad_x)$, $t \in \C$, preserves local finiteness, and differentiation at $t=0$ shows that the span of the locally finite elements is closed under brackets. Because Kac's argument is specific to $\C$, we give a proof valid over any field of characteristic zero.

 \begin{proposition}[Finite-dimensional Lie algebras generated by locally finite operators]\label{prop:fd-direction}
Let $\fa$ be a finite-dimensional Lie algebra over $\K$, and let $\rho:\fa\to\End_\K(V)$ be a representation.  If $\fa$ is generated as a Lie algebra by elements $x_1,\ldots,x_r$ such that each $\rho(x_i)$ is locally finite on $V$, then $\fa$ acts locally finitely on $V$.
\end{proposition}

\begin{proof}
Fix $v\in V$, and let $M:=\Span_\K\{\rho(x_{i_1})\cdots\rho(x_{i_s})v:s\ge0,\ 1\le i_j\le r\}\subseteq V$ be a  cyclic $\fa$-submodule of $V$. Since each $\rho(x_i)$ is locally finite on $V$, it is in particular locally finite on $M$.

Now define $\fh_M:=\{x\in\fa:\rho(x)\text{ is locally finite on }M\}$. We claim that $\fh_M$ is a Lie subalgebra of $\fa$. When $\K$ is algebraically closed, this is Proposition~\ref{prop:FK}; the general case follows by base change. Let $\Omega/\K$ be a field extension and $T\in\End_\K(M)$. If $T$ is locally finite and $w=\sum_{j=1}^q\omega_j\otimes v_j\in\Omega\otimes_\K M$, choose finite-dimensional $T$-stable subspaces $U_j\subseteq M$ containing the $v_j$; then the orbit of $w$ lies in $\Omega\otimes_\K(U_1+\cdots+U_q)$. The converse follows by applying local finiteness to $1\otimes v$ and using preservation of linear independence under scalar extension. Taking $\Omega=\overline\K$, an algebraic closure of $\K$, and writing $\overline\fa:=\overline\K\otimes_\K\fa$, the module $\overline\K\otimes_\K M$ is finitely generated as a $U(\overline\fa)$-module, so Proposition~\ref{prop:FK} shows that the locally finite elements of $\overline\fa$ form a Lie subalgebra; their preimage in $\fa$ is precisely $\fh_M$, which is therefore a Lie subalgebra of $\fa$, proving the claim. Since $x_1,\ldots,x_r\in\fh_M$ and these elements generate $\fa$ as a Lie algebra, it follows that $\fh_M=\fa$. Thus every element of $\fa$ acts locally finitely on $M$.

 Now choose a basis $E_1,\ldots,E_m$ of $\fa$. Since $M$ is a cyclic $\fa$-module, the Poincar\'e--Birkhoff--Witt theorem implies that $M$ is spanned by the ordered monomials $\rho(E_1)^{k_1}\cdots\rho(E_m)^{k_m}v$, with $k_1,\ldots,k_m\ge0$. For a finite-dimensional subspace $W\subseteq M$ and $1\le i\le m$, the subspace $V_i(W):=\Span_\K\{\rho(E_i)^kw:k\ge0,\ w\in W\}$ is finite-dimensional, because $\rho(E_i)$ is locally finite on $M$. Set $W_m:=V_m(\K v)$ and $W_i:=V_i(W_{i+1})$ for $i=m-1,\ldots,1$. Then each $W_i$ is finite-dimensional, and $\rho(E_1)^{k_1}\cdots\rho(E_m)^{k_m}v\in W_1$ for all $k_1,\ldots,k_m\ge0$: indeed, $\rho(E_m)^{k_m}v\in W_m$, then $\rho(E_{m-1})^{k_{m-1}}\rho(E_m)^{k_m}v\in W_{m-1}$, and so on. Hence $M\subseteq W_1$ is finite-dimensional.

Since $v\in V$ was arbitrary, every vector of $V$ is contained in a finite-dimensional $\fa$-stable subspace. Hence $\fa$ acts locally finitely on $V$.
\end{proof}

In the following result, $A$ need not be finitely generated.
\begin{proposition}[Finite-dimensional Lie algebras generated by locally finite subalgebras]\label{prop:kz-q23}
Let $\K$ be a field of characteristic zero, let $A$ be a $\K$-algebra, and let $\{\fh_i\}_{i\in\Omega}$ be a family of locally finite Lie subalgebras of $\Der_\K(A)$. If $\fh:=\Lie_\K\bigl(\bigcup_{i\in\Omega}\fh_i\bigr)$ is finite-dimensional, then $\fh$ is locally finite on $A$.
\end{proposition}
\begin{proof}
Since $\fh$ is finite-dimensional, a finite subset $S_0\subseteq\bigcup_{i\in\Omega}\fh_i$ generates $\fh$: choose a basis $v_1,\ldots,v_r$ of $\fh$, write each $v_i$ as a linear combination of iterated brackets of elements of the $\fh_i$,  and collect the finitely many elements appearing in the brackets. Every element of $S_0$ is locally finite on $A$, so Proposition~\ref{prop:fd-direction} applies.
\end{proof}
\noindent  This proposition also provides affirmative answers to Questions~2 and~3 of Kraft and Zaidenberg~\cite{KraftZaidenberg2024}, with no finite-generation hypothesis on $A$. For $A=\K[X]$, with $X$ an affine variety over an algebraically closed field,  the authors settled two particular cases: the first is when each $\fh_i$ is the Lie algebra of a connected algebraic subgroup of $\Aut(X)$~\cite[\S1.6, Corollary]{KraftZaidenberg2024}; the second is when $\fh$ is generated by locally nilpotent derivations~\cite[\S1.9, Corollary]{KraftZaidenberg2024}.

\xc{Faithful restriction.}
We record here an elementary observation that will be used repeatedly. A derivation of $A$ is determined by its values on any set of algebra generators; when $A$ is finitely generated, those generators lie in a finite-dimensional subspace. This observation, as well as  Corollary~\ref{cor:lf-subalg-fd} and Lemma~\ref{lem:adjoint-lf} below, is standard: for coordinate rings of affine varieties over an algebraically closed field these statements appear in~\cite[Lemmas~1.6.2 and~1.7.1]{KraftZaidenberg2024}, and in greater generality in~\cite[Remark~3.3(a) and Lemma~3.5]{ChitayatDaigleRegeta2026}. We include the short proofs to keep the paper self-contained.

\begin{lemma}[Faithful restriction]\label{lem:faithful-restriction}
Let $A$ be a commutative $\K$-algebra and let $E\subseteq A$ be a subspace containing a set of algebra generators of $A$. Then:
\begin{enumerate}
\item the restriction map
$$
\Der_\K(A)\longrightarrow\Hom_\K(E,A),\quad \eta\longmapsto\eta|_E,
$$
is injective;
\item if $E$ and $W\subseteq A$ are finite-dimensional, then $\{\eta\in\Der_\K(A):\eta(E)\subseteq W\}$ is finite-dimensional, of dimension at most $(\dim_\K E)(\dim_\K W)$;
\item if $E$ is stable under a Lie subalgebra $H\subseteq\Der_\K(A)$, then restriction gives an injective Lie algebra homomorphism $H\longrightarrow\End_\K(E)$.  In particular, if $E$ is finite-dimensional, then $\dim_\K H\leq(\dim_\K E)^2$.
\end{enumerate}
\end{lemma}

\begin{proof}
(1) If $\eta|_E=0$, then $\eta$ vanishes on a set of algebra generators of $A$, hence on all of $A$ by the Leibniz rule. (2) By~(1), the space in question is carried injectively into $\Hom_\K(E,W)$. (3) Since $E$ is $H$-stable, $[D,D']|_E=[D|_E,D'|_E]$ for all $D,D'\in H$, so restriction is a homomorphism of Lie algebras; it is injective by~(1).
\end{proof}

We isolate the consequence that will be used most often.

\begin{corollary}\label{cor:lf-subalg-fd}
Let $A$ be a finitely generated commutative $\K$-algebra, and let $H\subseteq\Der_\K(A)$ be a locally finite subalgebra. Then $H$ is finite-dimensional over $\K$.
\end{corollary}
\begin{proof}
Choose algebra generators $a_1,\ldots,a_N$ of $A$. Since $H$ is locally finite on $A$, for each $i$ there is a finite-dimensional $H$-stable subspace $U_i\subseteq A$ with $a_i\in U_i$. Then $U:=U_1+\cdots+U_N$ is finite-dimensional, $H$-stable, and contains the algebra generators, so Lemma~\ref{lem:faithful-restriction}(3) applies.
\end{proof}

Lemma~\ref{lem:faithful-restriction} has no analogue for arbitrary linear operators, as in general,  nothing requires a Lie algebra of operators to be determined by its restriction to a single finite-dimensional subspace. In fact, Theorem~\ref{thm:main} admits no counterpart in that setting: the following example exhibits two locally finite operators on a $\K$-vector space whose generated Lie algebra acts locally finitely yet is infinite-dimensional. Moreover, the adjoint actions of the two generators are nilpotent. Thus the rigidity phenomenon is specific to {\em derivations} of finitely generated commutative algebras.

\begin{example}\label{ex:not-derivations}
Let $V:=\bigoplus_{n\ge1}V_n$ with $V_n:=\K^2$, and let $e,f,h$ be the standard basis of $\mathfrak{sl}_2(\K)$, so that
$$
[e,f]=h,\quad [h,e]=2e,\quad [h,f]=-2f.
$$
Define $X,Y\in\End_\K(V)$ blockwise by
$$
X|_{V_n}:=e,\quad Y|_{V_n}:=nf,
$$
and set $L:=\Lie_\K(X,Y)$.

Every element of $L$ preserves each block $V_n$ and restricts on it to an element of $\mathfrak{sl}_2(\K)$. Since $(\ad_e)^3=(\ad_f)^3=0$ on $\mathfrak{sl}_2(\K)$, it follows that $(\ad_X)^3=(\ad_Y)^3=0$ on $L$.

Every $v\in V$ lies in the sum of the finitely many blocks supporting it, which is a finite-dimensional $L$-stable subspace. Hence $V$ is a locally finite $L$-module; in particular, $X$, $Y$, and every element of $L$ are locally finite on $V$.

Nevertheless, $L$ is infinite-dimensional. Indeed, $H:=[X,Y]$ satisfies $H|_{V_n}=nh$, whence
$$
(\ad_H)^m(X)|_{V_n}=2^m n^m e \quad\text{for all }m \geq 0.
$$
The operators $(\ad_H)^m(X)$ are linearly independent, so $L$ is infinite-dimensional. In particular, $\ad_H$ is not locally finite on $L$. Thus the analogues of conditions~(2) and~(3) of Corollary~\ref{cor:equiv} hold for the action of $L$ on $V$, while condition~(1) fails.

Concretely, {\em every} finite-dimensional subspace of $V$ lies in the $L$-invariant subspace $V_1\oplus\cdots\oplus V_N$ for some $N$. Finite-dimensional invariant subspaces are therefore readily available; but $L$ is infinite-dimensional, so it acts faithfully on none of them, in contrast with Lemma~\ref{lem:faithful-restriction}(3). \hfill\qed
\end{example}

\section{Adjoint local finiteness and finite-codimensional solvable ideals}\label{sec:adjoint}

\xc{Adjoint local finiteness.}
We now turn to the proof of Theorem~\ref{thm:main}. The first step is to show that when $A$ is a finitely generated $\K$-algebra, if $D \in \Der_\K(A)$ is  locally finite, then $\ad_D$ acts locally finitely on $\Der_\K(A)$:

\begin{lemma}[Adjoint local finiteness]
\label{lem:adjoint-lf}
Let $A$ be a finitely generated commutative $\K$-algebra. If $D\in\Der_\K(A)$ is locally finite on $A$, then $\ad_D:\Der_\K(A)\to\Der_\K(A)$ is locally finite.
\end{lemma}

\begin{proof}
Let $a_1,\ldots,a_r$ be algebra generators of $A$. From the local finiteness of $D$, we get a  finite-dimensional $D$-invariant subspace $E\subseteq A$ containing these generators; one may take $E:=\sum_{m=1}^{r}\Span_\K\{D^{i}(a_m):i\ge0\}$. 
Fix $\delta\in\Der_\K(A)$. The subspace $\delta(E)$ is finite-dimensional, so, appealing again to the local finiteness of $D$, we obtain a finite-dimensional $D$-invariant subspace $F\subseteq A$ with $\delta(E)\subseteq F$. 

Put $\mathcal V:=\{\eta\in\Der_\K(A):\eta(E)\subseteq F\}$. Since $E$ contains algebra generators of $A$ and $E$ and $F$ are finite-dimensional, $\mathcal V$ is finite-dimensional by Lemma~\ref{lem:faithful-restriction}(2). It is stable under $\ad_D$: if $\eta\in\mathcal V$, then, $E$ and $F$ being $D$-invariant,
$$
[D,\eta](E)\subseteq D(\eta(E))+\eta(D(E))\subseteq D(F)+\eta(E)\subseteq F.
$$
Since $\delta\in\mathcal V$, the entire $\ad_D$-orbit of $\delta$ lies in $\mathcal V$, so $\Span_\K\{(\ad_D)^m\delta:m\ge0\}$ is finite-dimensional. As $\delta$ was arbitrary, $\ad_D$ is locally finite on $\Der_\K(A)$.
\end{proof}

\xc{Finite-codimensional solvable ideals.}
We shall also use the following  fact:
\begin{lemma}\label{lem:finite-codim-ideal-fg}
Let $L$ be a finitely generated Lie algebra over a field $\K$, and let $M\triangleleft L$ be an ideal of finite codimension. Then $M$ is finitely generated as an ideal of $L$.
\end{lemma}

\begin{proof}
Let $x_1,\ldots,x_m$ generate $L$ as a Lie algebra, set $\pi: L \to Q:=L/M$ and let $e_1,\ldots,e_N$ be a $\K$-basis of $Q$. Let $y_i \in L$ be so that $\pi(y_i) = e_i$, $1 \leq i \leq N$. Let $c_{ij}^k, \alpha_{ai}\in \K$ be so that
$$
  [e_i,e_j]=\sum_k c_{ij}^k e_k
  \mbox{ and }
  \pi(x_a)= \sum_i \alpha_{ai}e_i.
$$
Set
$$
  t_{ij}:=[y_i,y_j]-\sum_k c_{ij}^k y_k \mbox{ and }
  s_a:=x_a-\sum_i\alpha_{ai}y_i.
$$
Then $t_{ij} \in M$ and $s_a \in M$. 

Let $J\triangleleft L$ be the ideal generated by the finitely many elements $t_{ij}$ and $s_a$. Then $J\subseteq M$. In $L/J$, the images $\overline y_i$ span a Lie subalgebra, since $[\overline y_i,\overline y_j]=\sum_k c_{ij}^k\overline y_k$. This subalgebra contains the images of the generators $x_a$, because $\overline x_a=\sum_i\alpha_{ai}\overline y_i$. Therefore, $\Span\{\overline y_i\}$ contains the subalgebra generated by the $\overline x_a$, which is $L/J$.

The quotient map $\rho:L/J\to L/M=Q$ sends $\overline y_i$ to $e_i$. Now if there exists $\lambda_i \in \K$ with $\sum_{i=1}^N \lambda_i \overline y_i=0$, then  $\sum_i \lambda_i \rho(\overline y_i)=\sum_i \lambda_i e_i=0$, which implies that $\lambda_i=0$, $1 \leq i \leq N$ since the $e_i$'s  are linearly independent. Thus the $\overline y_i$ form a basis of $L/J$, and the map $\rho$ is an isomorphism. Therefore $J=M$ and $M$ is generated as an ideal of $L$ by the finite set $\{t_{ij},s_a\}$.
\end{proof}

The following proposition handles the case of a finite-codimensional solvable ideal.
\begin{proposition}\label{prop:virtually-solvable}
Let $L$ be a finitely generated Lie algebra over $\K$ with generators $s_1,\ldots,s_m$ and suppose that
\begin{enumerate}
\item the operator $\ad_{s_i}$ is locally finite on $L$, for $1 \leq i \leq m$;
\item $L$ contains a solvable ideal   $J\triangleleft L$ of finite codimension.
\end{enumerate}
Then $\dim_\K L<\infty$.
\end{proposition}
Finite generation and solvability alone do not force finite-dimensionality. To see this, note that the metabelian Lie algebra with basis $x,y_0,y_1,\ldots$ and brackets $[x,y_i]=y_{i+1}$ and $[y_i,y_j]=0$ is generated by $x$ and $y_0$ and is solvable, yet infinite-dimensional; here $\ad_x$ fails to be locally finite. For a realization by derivations of an affine variety, see~\cite[Example~2.22]{ArzhantsevZaidenberg2025}.

\begin{proof}
It suffices to show  $\dim_\K J<\infty$, since $L/J$ is finite-dimensional.  We do so as follows: we prove by induction that each quotient $L/J^{(r)}$ along the derived series is finite-dimensional. At each step, the finite codimensionality implies that the abelian quotient $J^{(r)}/J^{(r+1)}$ is finitely generated as an $L/J^{(r)}$-module, and the  local finiteness of the adjoint action makes this module locally finite; finite-dimensionality of the quotient then follows.

Let $J=J^{(0)}\supset J^{(1)}\supset\cdots\supset J^{(d)}=0$ be the derived series of $J$. Each term $J^{(r)}$ is an ideal of $L$. We  prove, by induction on $r$, that $L/J^{(r)}$  is finite-dimensional. For $r=0$, this is exactly the assumption that $J=J^{(0)}$  has finite codimension in $L$.

Assume  that $L/J^{(r)}$  is finite-dimensional.  Then $J^{(r)}$ has finite codimension in $L$.  Since $L$ is finitely generated, Lemma~\ref{lem:finite-codim-ideal-fg} shows that  $J^{(r)}$ is finitely generated as an ideal of $L$.

Consider the abelian quotient $J^{(r)}/J^{(r+1)}$.  Since both $J^{(r)}$ and $J^{(r+1)}$ are ideals of $L$, the adjoint action of $L$ on $J^{(r)}$ descends to an action on $J^{(r)}/J^{(r+1)}$.  Moreover, $J^{(r)}$ acts trivially on this quotient, because $[J^{(r)},J^{(r)}]=J^{(r+1)}$. Hence the adjoint action factors through the finite-dimensional Lie algebra $L/J^{(r)}$ and we have the exact sequence of $L$-modules:
 \begin{equation}\label{eq:exseq-J}
 0\to J^{(r)}/J^{(r+1)}
  \to L/J^{(r+1)}
  \to L/J^{(r)}
  \to 0.
\end{equation}
We now show that $\dim_\K J^{(r)}/J^{(r+1)}<\infty$. To this end, let $y_1,\ldots,y_p$ be a generating set for $J^{(r)}$ as an ideal of $L$.  Then the classes $\overline y_1,\ldots,\overline y_p$ generate $J^{(r)}/J^{(r+1)}$ as an $L$-module.  Since the $L$-action factors through $L/J^{(r)}$, the same classes generate $J^{(r)}/J^{(r+1)}$ as an $L/J^{(r)}$-module.   The quotient $L/J^{(r)}$ is finite-dimensional by induction hypothesis, say it is generated by the images of $s_1,\ldots,s_m \in L$.  The induced actions of these generators on $J^{(r)}/J^{(r+1)}$ are locally finite, since local finiteness descends to quotients.  Hence Proposition~\ref{prop:fd-direction} implies that $L/J^{(r)}$ acts locally finitely on $J^{(r)}/J^{(r+1)}$. Since this module is finitely generated and the action is locally finite, the sum of finite-dimensional invariant subspaces containing the finitely many generators contains the module, and therefore $\dim_\K J^{(r)}/J^{(r+1)}<\infty$. From the exact sequence~\eqref{eq:exseq-J}, it follows that $L/J^{(r+1)}$ is finite-dimensional.  This completes the induction.

Taking $r=d$, we get $L/J^{(d)}=L$ finite-dimensional, since $J^{(d)}=0$.
\end{proof}

We conclude this section by proving Theorem~\ref{thm:mainsolvable}, stated in the introduction.

\begin{proof}[Proof of Theorem~\ref{thm:mainsolvable}]
For each $i$, Corollary~\ref{cor:lf-subalg-fd} shows that $\fh_i$ is finite-dimensional. Choose a basis of each $\fh_i$ and write the collection of the basis elements as $\{s_1,\ldots,s_m\}$. Then $s_1,\ldots,s_m$ are locally finite derivations of $A$ and they generate $\fh$ as a Lie algebra.

By Lemma~\ref{lem:adjoint-lf}, each $\ad_{s_j}$ is locally finite on $\Der_\K(A)$, hence also on the subalgebra $\fh$. Since $\fh$ is solvable and generated by $s_1,\ldots,s_m$, Proposition~\ref{prop:virtually-solvable}, applied with $J:=\fh$, gives $\dim_\K\fh<\infty$.
Finally, the finite-dimensional Lie algebra $\fh$ acts on $A$ and is generated by the locally finite operators $s_1,\ldots,s_m$. Proposition~\ref{prop:fd-direction} implies that $\fh$ is locally finite on $A$.
\end{proof}


\section{Cofinite separating ideals and the filtration by vanishing order}\label{sec:vanishing-order}

We now come to the geometric part of the proof: we isolate a finite-codimensional subalgebra of $\Der_\K(A)$, consisting of derivations vanishing to second order along a finite closed subscheme meeting the closure of every associated point of $\Spec A$, on which local finiteness forces local nilpotence. This will convert the rigidity question into a {\em solvability} problem for locally nilpotent derivations, taken up in Section~\ref{sec:LNDAA}. Throughout this section, $A$ is a non-zero finitely generated commutative $\K$-algebra.

\xc{Cofinite separating ideals.}
We say that $I$ is a {\em cofinite separating ideal} of $A$ if
\begin{equation}\label{eq:defcosepideal}
\mbox{(1)} \dim_\K(A/I) < \infty; 	\quad\mbox{ and } \quad \mbox{(2)} \bigcap_{q \geq 1}I^q = 0.
\end{equation}

\begin{lemma}\label{lem:existencecosepideal}
Every finitely generated commutative $\K$-algebra has a cofinite separating ideal. Moreover, if $I$ is a cofinite ideal of $A$, then $\dim_\K(A/I^q)<\infty$ for every $q\geq1$.
\end{lemma}
\begin{proof}
We first prove the second assertion. Since $A$ is noetherian, each $I^j/I^{j+1}$ is finitely generated as an $A/I$-module. Since  $A/I$ is finite-dimensional over $\K$, then so is $I^j/I^{j+1}$, and the exact sequences
$$
0\longrightarrow I^j/I^{j+1}\longrightarrow A/I^{j+1}\longrightarrow A/I^j\longrightarrow0
$$
show by induction that $A/I^q$ is finite-dimensional for every $q\geq1$.

For the existence statement, the proof is constructive. The finite set $\Ass(A)$ is nonempty when $A\neq 0$. For each $\fp\in\Ass(A)$ choose a maximal ideal $\fm_\fp\supseteq\fp$, and set
\begin{equation}\label{eq:defasspointid}
I:=\bigcap_{\fp\in\Ass(A)}\fm_\fp.
\end{equation}
 The Chinese remainder theorem gives $A/I\cong\prod_{i=1}^{k}A/\fm_i$, where $\fm_1,\ldots,\fm_k$ denote the distinct ideals among the $\fm_\fp$, and Zariski's lemma shows that $A/I$ is finite-dimensional; thus $I$ satisfies (1). For (2),  let $x\in\bigcap_{q\geq1}I^q$. By the Krull intersection theorem~\cite[Chap.~10]{AtiyahMacdonald}, there is $a\in I$ with $(1-a)x=0$. Recall that the set of zero divisors of $A$ equals $\cup_{\fp \in \Ass(A)}\fp$~\cite[Proposition~4.7]{AtiyahMacdonald}. If $1-a$ belonged to an associated prime $\fp$, then  $a+(1-a)=1\in I+\fp\subseteq\fm_\fp$, a contradiction. Hence $1-a$ is a nonzero divisor, and $x=0$.
\end{proof}

The use of associated primes in~\eqref{eq:defasspointid} is only to obtain the separatedness condition. In fact, for an ideal $I$ in a noetherian ring $A$, Krull intersection gives $\bigcap_{q\ge1}I^q=0$ if and only if $I+\fp\ne A$ for every $\fp\in\Ass(A)$. Equivalently, $V(I)$ meets the closure of every associated point of $\Spec A$. Thus, when $A$ is reduced, it is enough to choose closed points on the irreducible components; in the nonreduced case, embedded associated points must also be met, as the following example shows.

\begin{example}\label{ex:embedded-needed}
Let $A:=\K[x,\epsilon]/(\epsilon^2,(x-1)\epsilon)$. The reduced quotient is $\K[x]$, so there is one minimal prime, namely the nilradical $(\epsilon)$, and one irreducible component. Choosing the closed point $x=0$, with maximal ideal $\fm:=(x,\epsilon)$, one has $x\epsilon=\epsilon$, so $\epsilon=x^n\epsilon\in\fm^n$ for every $n\ge1$. Thus $0\ne\epsilon\in\bigcap_{n\ge1}\fm^n$, and separation fails. The embedded associated prime $(x-1,\epsilon)$ must also be used.
 \hfill \qed \end{example}

\xc{The filtration by vanishing order.}
From now on, $I$ denotes a cofinite separating ideal of $A$.  We use the convention $I^0:=A$. For $p\ge1$, set
\begin{equation}\label{eq:defDIp}
\mathcal D_I^p:=\{D\in\Der_\K(A):D(A)\subseteq I^p\},
\end{equation}
and set $\mathcal D_I^0:=\Der_\K(A)$. We have the following result:
\begin{lemma}\label{lem:vanishing-order-filtration}
For every $p\ge1$, the space $\mathcal D_I^p$ has finite codimension in $\Der_\K(A)$. If $D\in\mathcal D_I^p$ and $q\ge1$, then $D(I^q)\subseteq I^{p+q-1}$. Consequently, $[\mathcal D_I^p,\mathcal D_I^q]\subseteq\mathcal D_I^{p+q-1}$ for $p,q\geq 1$.
\end{lemma}

\begin{proof}
Choose algebra generators $a_1,\ldots,a_N$ of $A$ and consider the evaluation map
$$
  \Phi:\Der_\K(A)\to(A/I^p)^{\oplus N},\quad \delta\mapsto(\delta(a_1),\ldots,\delta(a_N))\bmod I^p.
$$
Then $\ker\Phi=\mathcal D_I^p$: indeed, if $\delta(A)\subseteq I^p$, then clearly $\Phi(\delta)=0$. Conversely, if $\delta(a_i)\in I^p$ for every $i$, then the Leibniz rule implies that $\delta(f(a_1,\ldots,a_N))\in I^p$ for every polynomial $f$, and hence $\delta(A)\subseteq I^p$. As $A/I^p$ is finite-dimensional by Lemma~\ref{lem:existencecosepideal}, the codomain of $\Phi$ is finite-dimensional and thus $\mathcal D_I^p$ has finite codimension in $\Der_\K(A)$.

The second and third assertions follow immediately from the Leibniz rule and~\eqref{eq:defDIp}.
\end{proof}

The subalgebra $\mathcal D_I^2$ consists of the derivations vanishing to second order along the finite closed subscheme $\Spec(A/I)$. It is a finite-codimensional Lie subalgebra of $\Der_\K(A)$ and obeys
$$
D(A)\subseteq I^2,\quad D(I^q)\subseteq I^{q+1}\quad(q\ge1)
$$
for every $D\in\mathcal D_I^2$.

We note that from Lemma~\ref{lem:vanishing-order-filtration}, for $D_1,\ldots,D_s\in\mathcal D_I^2$ one has  $D_s\cdots D_1(A)\subseteq I^{s+1}$ and the lower central series obeys $\gamma_s(\mathcal D_I^2)\subseteq\mathcal D_I^{s+1}$. Consequently, $\bigcap_{s\ge1}\gamma_s(\mathcal D_I^2)=0$. If $\fg\subseteq\mathcal D_I^2$ is finite-dimensional, its lower central series stabilizes; the stable term is contained in every $\mathcal D_I^{s+1}$ and is therefore zero. In particular, {\em every finite-dimensional Lie subalgebra of $\mathcal D_I^2$ is nilpotent}.

\xc{Local finiteness forces local nilpotence.}
The estimate above also yields the following key consequence.

\begin{lemma}[Local finiteness forces local nilpotence]\label{lem:second-order-lnd}
Let $D\in\mathcal D_I^2$. If $D$ is locally finite on $A$, then $D$ is locally nilpotent on $A$. In particular, $$
\LFD(A)\cap\mathcal D_I^2=\LND(A)\cap\mathcal D_I^2.
$$
\end{lemma}

\begin{proof}
For every $a\in A$ and $n\ge1$, one has $D^n(a)\in I^{n+1}$. Fix $a\in A$ and set $U:=\Span_\K\{D^n(a):n\ge0\}$.  The space $U$ is finite-dimensional by local finiteness. Hence the descending chain $U\cap I\supseteq U\cap I^2\supseteq\cdots$ stabilizes, and its stable value is $U\cap\bigcap_{q\ge1}I^q=0$, because $I$ is a cofinite separating ideal. Thus $U\cap I^N=0$ for some $N$ and, consequently,  $D^{N}(a)=0$. 
\end{proof}

\section{Proof of the rigidity theorem}\label{sec:mainproof}

The finite-codimensional subalgebra we will exhibit below is not, in general, an ideal of the generated Lie algebra. The following lemma of Petravchuk upgrades such a subalgebra to an ideal (Petravchuk uses the term \emph{almost solvable} for a Lie algebra containing a solvable ideal of finite codimension).

\begin{lemma}[Finite-codimensional solvable ideals, {\cite[Lemma~2(i),(ii)]{Petravchuk1999}}]\label{lem:Petravchuk}
Let $L$ be a Lie algebra over a field. If $L$ contains a solvable subalgebra of finite codimension, then $L$ contains a solvable ideal of finite codimension. In fact, one may choose such an ideal to be characteristic.
\end{lemma}
\begin{proof}
 By~\cite[Lemma~2(ii)]{Petravchuk1999}, the existence of a solvable subalgebra of finite codimension implies that $L$ is almost solvable. By~\cite[Lemma~2(i)]{Petravchuk1999}, an almost solvable Lie algebra contains a solvable characteristic ideal of finite codimension.
\end{proof}

We now can prove our main theorem:
\begin{proof}[Proof of Theorem~\ref{thm:main}]
Let $L\subseteq\LFD(A)$ be a finitely generated Lie subalgebra. The assertion is immediate when $A=0$, so we assume that $A\ne0$.

Let $I\subseteq A$ be a cofinite separating ideal of $A$, which exists by Lemma~\ref{lem:existencecosepideal}. By Lemma~\ref{lem:vanishing-order-filtration}, $\mathcal D_I^2$ has finite codimension in $\Der_\K(A)$, and therefore
$$
H_I:=L\cap\mathcal D_I^2
$$
is a Lie subalgebra of $L$ of finite codimension. Every $D\in H_I$ lies in $\LFD(A)$, so Lemma~\ref{lem:second-order-lnd} gives $D\in\LND(A)$. Hence $H_I\subseteq\LND(A)$, and Theorem~\ref{thm:lnd-solvable} shows that $H_I$ is solvable.

By the above, $H_I$ is a finite-codimensional solvable subalgebra of $L$, and thus using Lemma~\ref{lem:Petravchuk}, there exists a finite-codimensional solvable ideal $J\triangleleft L$. By Lemma~\ref{lem:adjoint-lf}, the generators of $L$ have locally finite adjoint action on $L$. Proposition~\ref{prop:virtually-solvable} therefore yields $\dim_\K L<\infty$.
\end{proof}

\begin{proof}[Proof of Corollary~\ref{cor:equiv}]
If $L\subseteq\LFD(A)$, then Theorem~\ref{thm:main} gives $(3)\Rightarrow(1)$. Proposition~\ref{prop:fd-direction} gives $(1)\Rightarrow(2)$, and $(2)\Rightarrow(3)$ is immediate.
\end{proof}

Corollary~\ref{cor:equiv} applies verbatim to the finitely generated subalgebras of an arbitrary Lie subalgebra $\fh\subseteq\LFD(A)$, and this yields a global form of Theorem~\ref{thm:main} with no finite-generation hypothesis on the Lie algebra.


\begin{proof}[Proof of Corollary~\ref{cor:global-lf}]
Let $\fh'\subseteq\fh$ be a finitely generated Lie subalgebra. Since $\fh'\subseteq\LFD(A)$, Theorem~\ref{thm:main} gives $\dim_\K\fh'<\infty$, and Corollary~\ref{cor:equiv} shows that $\fh'$ acts locally finitely on $A$. Hence $\fh$ is locally finite as an abstract Lie algebra. Taking $\fh'=\Lie_\K(S)$ for a finite subset $S\subseteq\fh$ gives the last assertion.
\end{proof}

We finally record the localized form of Theorem~\ref{thm:main} announced in the introduction.
\begin{theorem}\label{thm:intro-localized}
Let $A$ be a finitely generated commutative $\K$-algebra, with $\K$ of characteristic zero, let $L\subseteq\Der_\K(A)$ be generated by finitely many locally finite derivations, and let $I$ be a cofinite separating ideal of $A$. Put $H_I:=L\cap\mathcal D_I^2$. Then the following conditions are equivalent:
\begin{enumerate}
\item $\dim_\K L<\infty$;
\item $H_I$ is nilpotent;
\item $H_I$ is solvable;
\item $H_I\subseteq\LFD(A)$;
\item $H_I\subseteq\LND(A)$.
\end{enumerate}
\end{theorem}

\begin{proof}
By Lemma~\ref{lem:vanishing-order-filtration}, $H_I$ has finite codimension in $L$. If $L$ is finite-dimensional, then $H_I$ is a finite-dimensional Lie subalgebra of $\mathcal D_I^2$, and is therefore nilpotent by the observation following Lemma~\ref{lem:vanishing-order-filtration}. This gives $(1)\Rightarrow(2)$, while $(2)\Rightarrow(3)$ is immediate. If $H_I$ is solvable, the concluding argument in the proof of Theorem~\ref{thm:main} applies verbatim and gives $\dim_\K L<\infty$, proving $(3)\Rightarrow(1)$. Proposition~\ref{prop:fd-direction} gives $(1)\Rightarrow(4)$, Lemma~\ref{lem:second-order-lnd} gives $(4)\Rightarrow(5)$, and Theorem~\ref{thm:lnd-solvable} gives $(5)\Rightarrow(3)$.
\end{proof}

Since $\mathcal D_I^p\subseteq\mathcal D_I^2$ for $p\geq2$, and $L\cap\mathcal D_I^p$ has finite codimension in $L$ by Lemma~\ref{lem:vanishing-order-filtration}, the same proof shows that Theorem~\ref{thm:intro-localized} holds with $\mathcal D_I^2$ replaced by $\mathcal D_I^p$, for any $p\geq2$. In particular, if $L$ is infinite-dimensional, then for every $p\geq2$ it contains a derivation that is not locally finite and maps $A$ into $I^p$.

\section{Locally nilpotent derivations on affine algebras}\label{sec:LNDAA}

It remains to prove Theorem~\ref{thm:lnd-solvable}. Besides its  role in the proof of Theorem~\ref{thm:main},  it is of independent interest, as it bears on the problem of describing the Lie subalgebras of $\LND(A)$. After the domain case, passing to the irreducible components gives the reduced case. The main difficulty is the nonreduced case: we show that solvability survives the passage through the nilradical, and in fact that, for every Lie subalgebra $\fa\subseteq\LND(A)$, sufficiently long products of elements of $\fa$ acting trivially on $A/\mathcal N$ {\em vanish on $A$}, with a bound depending only on $A$. The key to this result is to keep track of products of operators, rather than of brackets, throughout the d\'evissage.

\xc{Preliminary facts.}
We will use repeatedly two simple representation-theoretic facts. The first is   that solvability is stable under extensions and finite products. For a homomorphism $\rho:\fg\to\fh$ of Lie algebras, the exact sequence
\begin{equation}\label{eq:solvable-extension}
0\longrightarrow\ker\rho\longrightarrow\fg\longrightarrow\rho(\fg)\longrightarrow0
\end{equation}
shows that $\fg$ is solvable whenever $\rho(\fg)$ and $\ker\rho$ are: if $\rho(\fg)^{(r)}=0$ then $\fg^{(r)}\subseteq\ker\rho$, and if moreover $(\ker\rho)^{(s)}=0$, we get $\fg^{(r+s)}=0$, showing that $\fg$ is solvable.  Moreover, if $\fg\subseteq\fh$, then $\fg^{(r)}\subseteq\fh^{(r)}$, while $ (\fh_1\times\cdots\times\fh_n)^{(r)} = \fh_1^{(r)}\times\cdots\times\fh_n^{(r)}$.  Together, these observations show that  every Lie subalgebra of a finite product of solvable Lie algebras is solvable. The second fact is that, in finite dimension, nilpotence of every element of a Lie subalgebra of $\End_\F(V)$ persists under extension of scalars:

\begin{lemma}\label{lem:nilpotent-scalar-extension}
Let $\K$ be an infinite field, let $\F\supseteq \K$ be a field extension, and let $V$ be a finite-dimensional $\F$-vector space. Let $\fs\subseteq\End_\F(V)$ be a $\K$-Lie subalgebra consisting of nilpotent endomorphisms. Then every element of the $\F$-linear span $\F\fs\subseteq\End_\F(V)$ is nilpotent. In particular, $\F\fs$ is a nilpotent Lie algebra, and every product of $\dim_\F V$ elements of $\F\fs$ is zero.
\end{lemma}

\begin{proof}
Let $X_1,\ldots,X_p\in\fs$  be a  basis of $\F\fs$, and set $X(t):=t_1X_1+\cdots+t_pX_p$. Let $n:=\dim_\F V$, and write the characteristic polynomial of $X(t)$ as $$
  \chi_{X(t)}(u):=\det(uI_V-X(t))
  =u^n+c_{n-1}(t_1,\ldots,t_p)u^{n-1}+\cdots+c_0(t_1,\ldots,t_p),
$$
where each $c_i\in \F[t_1,\ldots,t_p]$. For every $(\lambda_1,\ldots,\lambda_p)\in \K^p$, the operator $X(\lambda)=\lambda_1X_1+\cdots+\lambda_pX_p$ belongs to $\fs$, hence is nilpotent. Therefore $\chi_{X(\lambda)}(u)=u^n$, so $c_i(\lambda_1,\ldots,\lambda_p)=0$ for all $i=0,\ldots,n-1$. Since $\K$ is infinite, each polynomial $c_i$ vanishes identically. Hence for every $(a_1,\ldots,a_p)\in \F^p$, the operator $a_1X_1+\cdots+a_pX_p$ has characteristic polynomial $u^n$, and is therefore nilpotent by  Cayley--Hamilton. This implies that every element of $\F\fs$ is nilpotent. The $\F$-span $\F\fs$ is closed under brackets, because $\fs$ is closed under brackets and the commutator is $\F$-bilinear. It is finite-dimensional as a subspace of $\End_\F(V)$. Hence, by Engel's theorem, there  is an $\F$-basis of $V$ in which all elements of $\F\fs$ are strictly upper triangular, so that every product of $\dim_\F V$ of them vanishes and $\F\fs$ is nilpotent.
\end{proof}

We also record two simple facts about derivations. 
\begin{lemma}\label{lem:derivations-preserve-minimal-primes}
Let $B$ be a reduced finitely generated $\K$-algebra and let $\fp$ be a minimal prime of $B$. Then for all $D\in\Der_\K(B)$ we have $D(\fp)\subseteq\fp$.
\end{lemma}
This is standard, so we omit the proof.

The second is that derivations preserve the nilradical. Let $\mathcal N$ be the nilradical of a commutative $\K$-algebra $A$. Every derivation $D\in\Der_\K(A)$ preserves $\mathcal N$. Indeed, if $x^n=0$, then the Leibniz formula for $D^n(x^n)$ gives $0=n!D(x)^n+xp$ for some $p\in A$: the term in which every factor is differentiated once is $n!D(x)^n$, while every other term contains a factor $x$. Hence $D(x)^n\in xA$ and $D(x)^{n^2}=0$.

\xc{The domain and reduced cases.}  A theorem of Bezushchak, Petravchuk, and Zelmanov~\cite[Theorem~1.4]{BezushchakPetravchukZelmanov2024} says that every Lie subalgebra of $\LND(A)$ is {\em locally nilpotent as a Lie algebra}, in the sense that each of its finitely generated subalgebras is nilpotent. We will use this result together with a result of Makedonskyi and Petravchuk~\cite[Corollary~1]{MakedonskyiPetravchuk2014} stating that a nilpotent Lie subalgebra of $Q\Der_\K(C)$, where $Q:=\Frac(C)$, has derived length at most its rank over $Q$, that is, the dimension of its $Q$-linear span. It remains to bound this rank uniformly over the finitely generated subalgebras, and to pass from these subalgebras to the whole Lie algebra.

\begin{proposition}[Derived-length bound for domains]
\label{prop:domain-lnd-bound}
Let $C$ be a finitely generated commutative domain over a field $\K$ of characteristic zero, and let $\fa\subseteq\LND(C)$ be a Lie subalgebra. Put
$$
B:=\bigcap_{D\in\fa}\ker D,\quad  F:=\Frac(B),\quad\text{ and }\quad Q:=\Frac(C).
$$
Then $$\operatorname{dl}(\fa)\leq\dim_Q(Q\fa)\leq\trdeg_F Q,
$$ where $Q\fa$ denotes the $Q$-linear span of $\fa$ in $\Der_\K(Q)$.
\end{proposition}

\begin{proof}
Put $n:=\trdeg_F Q$ and $d:=\dim_Q(Q\fa)$. 
Let $\fh\subseteq\fa$ be a finitely generated Lie subalgebra. By the theorem of Bezushchak, Petravchuk, and Zelmanov, $\fh$ is nilpotent. Every element of $\fa$ annihilates $B$, hence also $F=\Frac(B)$, and therefore
$$
Q\fh\subseteq Q\fa\subseteq\Der_F(Q).
$$
Since $Q/F$ is finitely generated and separable, we have (see, e.g., \cite[Theorem~23.12]{Morandi})
$$
\dim_Q\Der_F(Q)=\trdeg_F Q=n.
$$
Thus $\operatorname{rk}_Q\fh:=\dim_Q(Q\fh)\leq d\leq n$, and the result of Makedonskyi and Petravchuk~\cite[Corollary~1]{MakedonskyiPetravchuk2014} quoted above gives
$$
\operatorname{dl}(\fh)\leq\operatorname{rk}_Q\fh\leq d\leq n.
$$
This bound is independent of $\fh$. Every element $u\in\fa^{(d)}$ belongs to $\fh^{(d)}$ for some finitely generated subalgebra $\fh\subseteq\fa$: this follows by induction on $d$, since every element of a derived algebra is a finite sum of brackets. Hence $u=0$, so $\fa^{(d)}=0$ and $\operatorname{dl}(\fa)\leq d=\dim_Q(Q\fa)\leq n=\trdeg_F Q$, as required.
\end{proof}

For an irreducible affine variety $X$ over an algebraically closed field of characteristic zero, Perepechko proves, by a different approach, that every Lie subalgebra of $\LND(\K[X])$ is solvable; see~\cite[Corollary~3.4]{Perepechko2023}. His argument, based on a triangularization theorem of Skutin~\cite{Skutin}, gives additional geometric information: after inverting a suitable element of the common kernel, $\K[X]$ becomes a polynomial ring over the localized kernel on which the Lie algebra acts by triangular derivations~\cite[Lemma~3.3]{Perepechko2023}. This triangularization, together with Lemma~\ref{lem:triangular-dl} below, also implies the bound $\operatorname{dl}(\fa)\leq\trdeg_F Q$ of Proposition~\ref{prop:domain-lnd-bound} in this setting.

At the group level, Kraft and Zaidenberg prove that nested unipotent subgroups of $\Aut(X)$ have derived length at most $\dim X$, and obtain the bound $\dim X+1$ for solvable algebraically generated groups and connected nested solvable groups; see~\cite[Theorem~5.3.1 and Corollary~5.3.2]{KraftZaidenberg2024}.

For a field $F$ and $n\geq0$, put
$$
T_n(F):=
F\partial_{x_1}\oplus
F[x_1]\partial_{x_2}\oplus\cdots\oplus
F[x_1,\ldots,x_{n-1}]\partial_{x_n},
$$
with $T_0(F)=0$; this is the Lie algebra of triangular derivations of $F[x_1,\ldots,x_n]$.

An elementary computation now shows that the bound of Proposition~\ref{prop:domain-lnd-bound} is sharp:

\begin{lemma}[Derived length of the triangular Lie algebra]\label{lem:triangular-dl}
For every field $F$ of characteristic zero, $T_n(F)\subseteq\LND(F[x_1,\ldots,x_n])$ and has derived length exactly $n$.
\end{lemma}

\begin{proof}
The statement is immediate for $n=0$, so assume $n\geq1$. Every element of $T_n(F)$ sends each variable $x_j$ into $F[x_1,\ldots,x_{j-1}]$ and is thus locally nilpotent. To compute the derived length of $T_n(F)$, set
$$
U_r:=\bigoplus_{i=r+1}^{n}F[x_1,\ldots,x_{i-1}]\partial_{x_i} \mbox{ for }0\leq r\leq n,
$$
with $U_n=0$. We claim that $[U_r,U_r]=U_{r+1}$ for $0\leq r<n$.
The inclusion $[U_r,U_r]\subseteq U_{r+1}$ follows from triangularity.  Conversely, every element of each summand of $U_{r+1}$ is a single bracket of elements of $U_r$, by integration with respect to $x_{r+1}$. Hence $T_n(F)^{(r)}=U_r$ for $0\le r\le n$. Since $U_{n-1}=F[x_1,\ldots,x_{n-1}]\partial_{x_n}\neq 0$ and $U_n=0$, the derived length of $T_n(F)$ is exactly $n$.
\end{proof}

\noindent Since $T_n(F)$ contains $\partial_{x_1},\ldots,\partial_{x_n}$, its common kernel in $F[x_1,\ldots,x_n]$ is $F$, and the bound of Proposition~\ref{prop:domain-lnd-bound} is attained in every transcendence degree. For the corresponding group, the proof of~\cite[Lemma~3.2]{FurterPoloni2018} shows that the group of unipotent triangular automorphisms of $\mathbb A^n_\C$ has derived length $n$.

\begin{proposition}\label{prop:reduced-lnd-solvable}
Let $R$ be a nonzero reduced finitely generated commutative $\K$-algebra with minimal primes $\fp_1,\ldots,\fp_r$ and let $\fa\subseteq\LND(R)$ be a Lie subalgebra. Define $C_i:=R/\fp_i$, let $\fa_i\subseteq\LND(C_i)$ be the image of $\fa$, and set
$$
B_i:=\bigcap_{D\in\fa_i}\ker D,\quad F_i:=\Frac(B_i),\mbox{ and }\, n_i:=\trdeg_{F_i}\Frac(C_i).
$$
Then the natural map $\fa\longrightarrow\prod_i\fa_i$ is injective and $\operatorname{dl}(\fa)\leq\max_i n_i\leq\dim R$.
\end{proposition}

\begin{proof}
Every derivation of $R$ preserves each $\fp_i$ by Lemma~\ref{lem:derivations-preserve-minimal-primes}, so the action of $\fa$ descends to the quotients $C_i$.  Since $R$ is reduced, the map $R\to\prod_iC_i$ is injective, and therefore so is $\fa\to\prod_i\fa_i$. Proposition~\ref{prop:domain-lnd-bound} gives $\operatorname{dl}(\fa_i)\leq n_i$ for every $i$. A Lie subalgebra of a finite product has derived length at most the maximum of the derived lengths of the factors, and $n_i\leq\trdeg_\K C_i\leq\dim R$.
\end{proof}

\xc{Nilpotent thickenings.}
The remaining work is to pass  to an arbitrary affine algebra.  Let $R$ be a finitely generated $\K$-algebra and let $N$ be a nonzero finitely generated $R$-module. We write $\Supp(N)$ for the support of $N$, i.e., the set of primes $\fp\in\Spec(R)$ with $N_\fp\neq0$, and $\MinSupp(N)$ for the set of minimal elements of $\Supp(N)$. For $\fp\in\MinSupp(N)$, the module $N_\fp$ is supported only at the closed point $\fp R_\fp$ of $\Spec R_\fp$ and therefore has finite length. 

We set
\begin{equation}\label{eq:def-generic-kernel}
K(N):=\ker\Bigl(N\longrightarrow\prod_{\fp\in\MinSupp(N)}N_\fp\Bigr),
\end{equation}
the submodule of elements of $N$ that vanish at every generic point of $\Supp(N)$. If $u\in\End_R(N)$, then $u(K(N))\subseteq K(N)$, since $(u(x))_\fp=u_\fp(x_\fp)$ for every prime $\fp$. The next lemma shows that $K(N)$ has a support of  {\em strictly smaller} dimension than $N$.

\begin{lemma}[Support dimension decrease]\label{lem:support-dimension-decrease}
Let $R$ be a finitely generated $\K$-algebra and let $N$ be a nonzero finitely generated $R$-module. Then $K(N)_\fp=0$ for every $\fp\in\MinSupp(N)$. Consequently, either $K(N)=0$ or
$$
\dim\Supp(K(N))<\dim\Supp(N).
$$
\end{lemma}

\begin{proof}
Fix $\fp\in\MinSupp(N)$ and let $\lambda:N\to N_\fp$ be the localization map. Since $N_\fp$ is already an $R_\fp$-module, $\lambda_\fp$ is the identity of $N_\fp$, and localization being exact, $(\ker\lambda)_\fp=\ker(\lambda_\fp)=0$. As $K(N)\subseteq\ker\lambda$, we get $K(N)_\fp=0$.

Assume now that $K(N)\ne0$ and let $\fq\in\Supp(K(N))$. Then $\fq\in\Supp(N)$, because $K(N)\subseteq N$, so $\fq$ contains some $\fp\in\MinSupp(N)$; moreover $\fq\ne\fp$, since $K(N)_\fp=0$. Thus $\fq/\fp$ is a nonzero prime of the finitely generated $\K$-domain $R/\fp$, and therefore
$$
\dim R/\fq<\dim R/\fp\leq\dim\Supp(N).
$$
As $\fq\in\Supp(K(N))$ was arbitrary, the  inequality follows.
\end{proof}

The following lemma is proved by a simple d\'evissage and will be applied repeatedly below:

\begin{lemma}\label{lem:filtered-devissage}
Let $V$ be a $\K$-vector space equipped with a finite decreasing filtration
$$
V=V^0\supseteq V^1\supseteq\cdots\supseteq V^s=0.
$$
Let $\mathcal S\subseteq\End_\K(V)$ be a set of operators preserving the filtration. Assume that, for every $0\le i\le s-1$, there is an integer $b_i\geq0$ such that every product of $b_i$ elements of $\mathcal S$ induces the zero map on $V^i/V^{i+1}$ (for $b_i=0$, this means that $V^i=V^{i+1}$). Then every product of $b_0+b_1+\cdots+b_{s-1}$ elements of $\mathcal S$ is zero on $V$.
\end{lemma}

\begin{proof}
Let $c_0=0$ and $c_i:=b_0+\cdots+b_{i-1}$ for $1\le i\le s$. We show by induction, for $D_1,\ldots,D_{c_s}\in\mathcal S$,  that $D_{c_i}\cdots D_1(V)\subseteq V^i$. The claim is clear for $i=0$; assume it holds for some $i<s$, then because the next $b_i$ factors preserve $V^i$ and their product induces the zero map on $V^i/V^{i+1}$, they map $V^i$ into $V^{i+1}$, and $D_{c_{i+1}}\cdots D_1(V)\subseteq D_{c_{i+1}}\cdots D_{c_i+1}(V^{i}) \subseteq V^{i+1}$. For $i=s$, we obtain $D_{c_s}\cdots D_1(V)\subseteq V^s=0$.
\end{proof}

\xc{Locally nilpotent module endomorphisms.}
We now prove the statement that controls the nilradical part of Theorem~\ref{thm:lnd-solvable}. The number of operators needed to annihilate a module will be measured by the following integer, which mirrors the induction on the dimension of the support used in the proof. Let $R$ be a finitely generated $\K$-algebra. For a finitely generated $R$-module $M$, set $\nu_R(M):=0$ if $M=0$. Owing to Lemma~\ref{lem:support-dimension-decrease}, we can define recursively for $M \neq 0$
$$
\nu_R(M):=\max_{\fp\in\MinSupp(M)}\operatorname{length}_{R_\fp}(M_\fp)+\nu_R\bigl(K(M)\bigr).
$$
As we observed above, the lengths are finite. Thus $\nu_R(M)$ is a nonnegative integer, which depends only on $R$ and $M$, and is positive when $M\neq0$. When $R$ is a domain and $M$ is torsion-free, $K(M)=0$, and $\nu_R(M)$ is simply the rank $\dim_{\Frac(R)}(\Frac(R)\otimes_RM)$ of $M$.

\begin{lemma}\label{lem:module-engel}
Let $R$ be a finitely generated commutative $\K$-algebra, let $M$ be a finitely generated $R$-module, and let $\fs\subseteq\End_R(M)$ be a $\K$-Lie subalgebra such that every element of $\fs$ is locally nilpotent on $M$. Then every product of $\nu_R(M)$ elements of $\fs$ is zero on $M$. In particular, $\fs$ is nilpotent.
\end{lemma}
We emphasize that $\fs$ is a Lie algebra of $R$-linear endomorphisms; the lemma therefore does not apply directly to Lie algebras of derivations, which are $\K$-linear but not $R$-linear.

We record one observation used throughout the proof. Since $M$ is finitely generated over $R$ and the operators in $\fs$ are $R$-linear, local nilpotence of  $D\in\fs$ on $M$ already implies its nilpotence on $M$: choosing $R$-generators $m_1,\ldots,m_N$ and integers $e_i$ with $D^{e_i}m_i=0$, one has $D^{e}M=0$ for $e=\max_i e_i$.

 When $R$ is a domain and $M$ is torsion-free, the proof of Lemma~\ref{lem:module-engel} is  simpler. Let $Q:=\Frac(R)$. The natural map $M\to Q\otimes_RM$ is injective since $M$ is torsion free, and therefore $\fs$ embeds, as a $\K$-Lie algebra, into $\End_Q(Q\otimes_RM)$. Since $M$ is finitely generated,  $\dim_Q (Q \otimes_R M)< \infty$.  By the observation above, each $D\in\fs$ is nilpotent on $M$, so the induced endomorphism of $Q\otimes_RM$ is nilpotent. By Lemma~\ref{lem:nilpotent-scalar-extension}, the $Q$-linear span of the image of $\fs$ in $\End_Q(Q\otimes_RM)$ consists of nilpotent endomorphisms and is a Lie subalgebra, and every product of $r:=\dim_Q(Q\otimes_RM)$ of its elements vanishes. Since $M$ embeds in $Q\otimes_RM$, every product of $r$ elements of $\fs$ is therefore zero on $M$, which is the assertion of the lemma, as $\nu_R(M)=r$ in this case.

In general, the passage to the fraction field is replaced by localization at the primes of $\MinSupp(M)$. For $\fp\in\MinSupp(M)$, the support of $M_\fp$ consists only of the maximal ideal $\fm_\fp:=\fp R_\fp$. Hence $\fm_\fp^kM_\fp=0$ for some $k\geq1$. Moreover, $\fm_\fp$ annihilates $\fm_\fp^jM_\fp/\fm_\fp^{j+1}M_\fp$, so this quotient is a vector space over the residue field $R_\fp/\fm_\fp$. Since $R_\fp$ is noetherian and $M_\fp$ is finitely generated, the quotient is finitely generated over $R_\fp$, and hence finite-dimensional over $R_\fp/\fm_\fp$.  Thus $M_\fp$ has finite length and the argument from the previous paragraph can then be applied to the action of $\fs$ on each successive quotient $\fm_\fp^jM_\fp/\fm_\fp^{j+1}M_\fp$.  What this localization does not see is precisely the submodule $K(M)$ of~\eqref{eq:def-generic-kernel}, whose support has smaller dimension by Lemma~\ref{lem:support-dimension-decrease}; an induction on $\dim\Supp(M)$ then concludes the proof. At each step, we keep track of the number of factors needed: these numbers add up along the filtrations involved, and this is how $\nu_R(M)$ arises.

\begin{proof}[Proof of Lemma~\ref{lem:module-engel}]

The assertion is trivial for $M=0$, so we assume that $M\neq0$ and argue by induction on $d:=\dim\Supp(M)$. Write
$$
\MinSupp(M)=\{\fp_1,\ldots,\fp_r\},\quad K:=K(M),
$$
and set
$$
\ell_j:=\operatorname{length}_{R_{\fp_j}}(M_{\fp_j})\quad\text{and}\quad\ell:=\max_{1\leq j\leq r}\ell_j,
$$
so that $\nu_R(M)=\ell+\nu_R(K)$.

We first show that every product of $\ell$ elements of $\fs$ maps $M$ into $K$. Fix $j$ and put $\fm_j:=\fp_jR_{\fp_j}$ and $Q_j:=R_{\fp_j}/\fm_j=\Frac(R/\fp_j)$. The module $M_{\fp_j}$ has finite length over $R_{\fp_j}$, so $\fm_j$ acts nilpotently on it: there is $e_j\geq1$ with
$$
M_{\fp_j}\supseteq\fm_jM_{\fp_j}\supseteq\cdots\supseteq\fm_j^{e_j}M_{\fp_j}=0,
$$
and the successive quotients of this filtration are finite-dimensional $Q_j$-vector spaces, whose dimensions sum to $\ell_j$. The localized operators are $R_{\fp_j}$-linear, hence preserve the filtration, and every element of $\fs$ is nilpotent on $M$ by the observation preceding the proof, hence induces a nilpotent endomorphism on each quotient. By Lemma~\ref{lem:nilpotent-scalar-extension}, applied to the image of $\fs$ on each quotient $W:=\fm_j^iM_{\fp_j}/\fm_j^{i+1}M_{\fp_j}$, every product of $\dim_{Q_j}W$ elements of $\fs$ induces the zero map on $W$. Lemma~\ref{lem:filtered-devissage} shows then that every product of $\ell_j$, and hence of $\ell$, elements of $\fs$ is zero on $M_{\fp_j}$. Since the elements of $\fs$ are $R$-linear, localization commutes with their products: for $D_1,\ldots,D_\ell\in\fs$ and $a\in M$, the image of $D_\ell\cdots D_1a$ in $M_{\fp_j}$ is $(D_\ell)_{\fp_j}\cdots(D_1)_{\fp_j}(a_{\fp_j})=0$, for every $j$. By~\eqref{eq:def-generic-kernel}, this means that $D_\ell\cdots D_1(M)\subseteq K$.

The submodule $K$ is preserved by $\fs$, as observed below~\eqref{eq:def-generic-kernel}; it is finitely generated because $R$ is noetherian, and every element of $\fs$ remains locally nilpotent on it. By Lemma~\ref{lem:support-dimension-decrease}, either $K=0$ or $\dim\Supp(K)<d$. In both cases, every product of $\nu_R(K)$ elements of $\fs$ is zero on $K$: trivially in the first case, and by the induction hypothesis, applied to the image of $\fs$ in $\End_R(K)$, in the second.

Combining the two steps, every product of $\ell+\nu_R(K)=\nu_R(M)$ elements of $\fs$ is zero on $M$. Finally, every iterated commutator of length $\nu_R(M)$ of elements of $\fs$ is a linear combination of such products, so that $\gamma_{\nu_R(M)}(\fs)=0$, and $\fs$ is nilpotent.
\end{proof}

We are now in a position to prove Theorem~\ref{thm:lnd-solvable}; we do so by proving a stronger statement that gives a bound, depending only on $A$, on the nilpotence class of the kernel of the action on $A/\sqrt{0}$.

\begin{proposition}\label{prop:uniform-kernel}
Let $A$ be a finitely generated commutative $\K$-algebra with nilradical $\mathcal N$. There is an integer $b_A\geq1$, depending only on $A$, with the following property: for every Lie subalgebra $\fa\subseteq\LND(A)$, the ideal $\fk:=\{D\in\fa:D(A)\subseteq\mathcal N\}$ of $\fa$ satisfies
$$
D_{b_A}\cdots D_1(A)=0\qquad\text{for all }D_1,\ldots,D_{b_A}\in\fk.
$$
In particular, $\gamma_{b_A}(\fk)=0$. If $A\neq0$, then $\fa^{(\dim A)}\subseteq\fk$ and $\operatorname{dl}(\fa)\leq\dim A+\lceil\log_2b_A\rceil$. If $A$ is reduced, one can take $b_A=1$.
\end{proposition}

\begin{proof}[Proof of Theorem~\ref{thm:lnd-solvable} and Proposition~\ref{prop:uniform-kernel}]
Let $\fa\subseteq\LND(A)$ be a Lie subalgebra. The assertion is a trivial consequence of the assumptions when $A=0$, so we assume that $A\neq0$. Recall that $\mathcal N:=\sqrt{0}$ is the nilradical of $A$, and choose $s$ to be the {\em least integer} such that $\mathcal N^s=0$, and set $R:=A/\mathcal N$. Since derivations preserve the nilradical, every element of $\fa$ induces a derivation of $R$, and we get a homomorphism $$\pi:\fa\to\Der_\K(R).$$ Moreover, $\pi(\fa)\subseteq\LND(R)$, since local nilpotence descends to quotients.

The image $\pi(\fa)$ is a Lie subalgebra of $\LND(R)$, and is therefore solvable by Proposition~\ref{prop:reduced-lnd-solvable} and, more precisely, $\operatorname{dl}(\pi(\fa))\leq\dim R=\dim A$, that is, $\fa^{(\dim A)}\subseteq\ker\pi$.

It remains to prove that $\fk:=\ker\pi$ is solvable; to this end, we will prove the stronger statement given in Proposition~\ref{prop:uniform-kernel}. By definition, $\fk=\{D\in\fa:D(A)\subseteq\mathcal N\}$. We apply Lemma~\ref{lem:filtered-devissage} to the nilradical filtration
$$
A\supseteq\mathcal N\supseteq\mathcal N^2\supseteq\cdots\supseteq\mathcal N^s=0.
$$
Every $D\in\fk$ preserves this filtration since, as already mentioned,  derivations preserve $\mathcal N$ and hence preserve all powers of $\mathcal N$. The image of $\fk$ on $A/\mathcal N$ is zero. For $1\le i\le s-1$, set $M_i:=\mathcal N^i/\mathcal N^{i+1}$. Since $A$ is noetherian and $\mathcal N M_i=0$, each $M_i$ is a finitely generated $R=A/\mathcal N$-module.

We claim that the image of $\fk$ in $\End_\K(M_i)$ lies in $\End_R(M_i)$ and satisfies the hypotheses of Lemma~\ref{lem:module-engel}. Let $D\in\fk$. It induces a $\K$-linear endomorphism of $M_i$ and, even more, an $R$-linear endomorphism of $M_i$: If $a\in A$ and $m\in\mathcal N^i$, then $D(am)=D(a)m+aD(m)\equiv aD(m)\pmod{\mathcal N^{i+1}}$, because $D(a)\in\mathcal N$ and hence $D(a)m\in\mathcal N^{i+1}$. Thus the induced endomorphism of $M_i$ is $R$-linear. It is locally nilpotent because $D$ is locally nilpotent on $A$ and local nilpotence descends to quotients. Therefore the image of $\fk$ in $\End_R(M_i)$ is a Lie algebra of locally nilpotent $R$-linear endomorphisms of the finitely generated $R$-module $M_i$. By Lemma~\ref{lem:module-engel}, every product of $\nu_R(M_i)$ elements of $\fk$ induces the zero map on $M_i$.

Since $\fk$ acts trivially on $A/\mathcal N$, we can apply Lemma~\ref{lem:filtered-devissage} to the nilradical filtration with $b_0:=1$ and $b_i:=\nu_R(M_i)$ for $1\le i\le s-1$. It shows that every product of
\begin{equation*}
b_A:=1+\sum_{i=1}^{s-1}\nu_R(\mathcal N^i/\mathcal N^{i+1})
\end{equation*}
elements of $\fk$ is zero on $A$. This integer depends only on $A$.  Every iterated commutator of length $b_A$ of elements of $\fk$ is a linear combination of such products, so that $\gamma_{b_A}(\fk)=0$ and $\fk$ is nilpotent. Likewise, every element of $\fk^{(q)}$ is a linear combination of products of $2^q$ elements of $\fk$, so that $\fk^{(q)}=0$ for $q:=\lceil\log_2b_A\rceil$. Finally, $\fa$ is an extension of the solvable Lie algebra $\pi(\fa)$ by the solvable ideal $\ker\pi$, so $\fa$ is solvable by~\eqref{eq:solvable-extension}. In more detail, we get $\fa^{(\dim A+q)}\subseteq\fk^{(q)}=0$, that is, $\operatorname{dl}(\fa)\leq\dim A+\lceil\log_2b_A\rceil$, and $\fa^{(\dim A)}\subseteq\fk$ is nilpotent, of class at most $b_A-1$.

If $A$ is reduced, then $\mathcal N=0$, so that $\fk=0$ and $b_A=1$, and Proposition~\ref{prop:reduced-lnd-solvable} gives $\operatorname{dl}(\fa)\leq\dim A$.  This concludes the proof.
\end{proof}
The example following the statement of Theorem~\ref{thm:lnd-solvable} in the introduction shows that the bound $\operatorname{dl}(\fa)\leq\dim A$ requires reducedness.

\section{Consequences and applications}\label{sec:consequences}

We now derive the applications of the rigidity theorem: algebraicity criteria for subgroups generated by algebraic subgroups, Popov's two-generator problem, and exact bilinear realizations of control systems. Throughout this section, $A$ denotes a finitely generated commutative $\K$-algebra.

\subsection*{Algebraicity of generated automorphism groups}

We now prove Corollary~\ref{cor:algebraicity-generated-subgroups}, stated in the introduction.

\begin{proof}[Proof of Corollary~\ref{cor:algebraicity-generated-subgroups}]
Each $\Lie(G_i)$ is contained in $\LFD(A)$. Choosing bases of the finitely many $\Lie(G_i)$ therefore gives finitely many locally finite derivations generating $L$, so Corollary~\ref{cor:equiv} gives the equivalence of (2)--(4). The equivalence with (1), and $\Lie(G)=L$, are Kraft and Zaidenberg's criterion~\cite[Theorem~A]{KraftZaidenberg2024}.
\end{proof}

\begin{corollary}[Solvable Lie algebras generated by locally nilpotent derivations]\label{cor:solvable-generated-lnd}
Let $D_1,\ldots,D_m\in\LND(A)$ and put $L:=\Lie_\K(D_1,\ldots,D_m)$. Then
$$
L\text{ is solvable}\quad\Longleftrightarrow\quad L\subseteq\LND(A).
$$
When these conditions hold, $L$ is finite-dimensional and nilpotent.
\end{corollary}

\begin{proof}
If $L\subseteq\LND(A)$, Theorem~\ref{thm:lnd-solvable} implies that $L$ is solvable.  Conversely, suppose that $L$ is solvable. Each $\K D_i$ is a locally finite subalgebra of $\Der_\K(A)$, so Theorem~\ref{thm:mainsolvable} shows that $L$ is finite-dimensional and acts locally finitely on $A$. Choose a finite-dimensional $L$-stable subspace $U\subseteq A$ containing algebra generators; by Lemma~\ref{lem:faithful-restriction}(3), restriction gives an injective homomorphism of Lie algebras $L\longrightarrow\End_\K(U)$. Write $\overline L:=\overline\K\otimes_\K L$ and $\overline U:=\overline\K\otimes_\K U$, where $\overline\K$ is an algebraic closure of $\K$. Lie's theorem simultaneously upper triangularizes the action of $\overline L$ on $\overline U$.  Each $D_i$ is nilpotent on $U$, hence acts on $\overline U$ by a {\em strictly} upper triangular matrix. The $D_i$ generate $\overline L$ as a $\overline\K$-Lie algebra, and the strictly upper triangular matrices form a Lie algebra, so every element of $\overline L$ acts on $\overline U$ by a strictly upper triangular matrix. Restriction to $\overline U$ is still faithful on $\overline L$, so $\overline L$ is nilpotent, and hence so is $L$; moreover every $D\in L$ is nilpotent on $U$.

Since $U$ contains algebra generators of $A$, the Leibniz rule then shows that every $D\in L$ is locally nilpotent on $A$. Therefore $L\subseteq\LND(A)$.
\end{proof}

\subsection*{Popov's problem on two unipotent one-parameter subgroups}

We now specialize Corollary~\ref{cor:algebraicity-generated-subgroups} to its smallest nontrivial instance: $X=\mathbb A^n$ and $m=2$, with $G_1$ and $G_2$ the one-parameter unipotent subgroups generated by two locally nilpotent derivations. This case belongs to the classical study of $\mathbb G_a$- and unipotent actions on affine space~\cite{Rentschler1968,Popov1987,Snow1989}, and it is the problem posed by Popov in 2005~\cite[Problem~3.1]{PopovProblems2005} (recorded also as~\cite[Question~11.30]{Freudenburg2006}) and stated in the introduction. The specialization yields the following intrinsic criterion; for the ind-group structure on $\Aut(\mathbb A^n)$ and its topology, see~\cite[\S 1]{KraftZaidenberg2024}.

\begin{corollary}\label{cor:popov}
Let $\K$ be algebraically closed of characteristic zero, let $D,E\in\LND(\K[x_1,\ldots,x_n])$, and let $G$ be the minimal subgroup of $\Aut(\mathbb A^n)$, closed in the ind-group topology, containing the two one-parameter unipotent subgroups $
  \{\exp(tD):t\in\K\}$ and $\{\exp(tE):t\in\K\}$. Put $L:=\Lie_\K(D,E)\subseteq\Der_\K(\K[x_1,\ldots,x_n])$.
Then
$$
\dim G<\infty\iff\dim_\K L<\infty\iff L\subseteq\LFD(\K[x_1,\ldots,x_n]).
$$
\end{corollary}
\begin{proof}[Proof of Corollary~\ref{cor:popov}]
Write $G_D:=\{\exp(tD):t\in\K\}$ and $G_E:=\{\exp(tE):t\in\K\}$, and let $G_0:=\langle G_D,G_E\rangle$ be the subgroup they generate. For every $g\in G_0$, there exist $r\geq1$ and $i_1,\ldots,i_r\in\{D,E\}$ such that $g$ belongs to the algebraic
family
$$
\varphi_{T}
:=\exp(t_1i_1)\cdots\exp(t_ri_r),
\mbox{ with }
T=(t_1,\ldots,t_r)\in\mathbb A^r.
$$
This family is contained in $G_0$ and contains the identity at $T=0$. Hence $G_0$ is a connected subgroup of $\Aut(\mathbb A^n)$ in the sense of~\cite{Ramanujam1964}; see also~\cite[\S2]{Popov2014}. If $\dim G<\infty$, then $G_0\subseteq G$ is finite-dimensional as well, so Ramanujam's theorem~\cite{Ramanujam1964}  makes $G_0$ a connected algebraic subgroup of $\Aut(\mathbb A^n)$. By the universal property in Ramanujam's theorem, the families $t\mapsto\exp(tD)$ and $t\mapsto\exp(tE)$ define morphisms $\mathbb G_a\to G_0$; hence $D,E\in\Lie(G_0)$ and $L\subseteq\Lie(G_0)$ is finite-dimensional. Conversely, if $L$ is finite-dimensional, Corollary~\ref{cor:algebraicity-generated-subgroups} shows that $G_0$ is algebraic, hence closed, and since $G$ is the minimal closed subgroup containing the same one-parameter groups, $G=G_0$. The equivalence between finite-dimensionality of $L$ and  $L\subseteq\LFD(\K[x_1,\ldots,x_n])$ is Corollary~\ref{cor:equiv}.
\end{proof}

\subsection*{Embeddings of nonlinear control systems}
We now derive the consequences for the super-linearization of control-affine systems.
\begin{corollary}[Super-linearization criterion]\label{cor:superlinearization}
Let $\K\in\{\mathbb R,\mathbb C\}$, let $f_0,\ldots,f_r$ be polynomial vector fields on $\K^n$ with associated derivations $D_0,\ldots,D_r$ of $\K[x_1,\ldots,x_n]$, and set $L:=\Lie_\K(D_0,\ldots,D_r)$. Then the following are equivalent:
\begin{enumerate}
\item there is a finite-dimensional subspace $U\subseteq\K[x_1,\ldots,x_n]$ containing $1,x_1,\ldots,x_n$ and satisfying $D_i(U)\subseteq U$ for all $0\le i\le r$;
\item $L\subseteq\LFD(\K[x_1,\ldots,x_n])$;
\item $D_i\in\LFD(\K[x_1,\ldots,x_n])$ for every $i$ and $\dim_\K L<\infty$.
\end{enumerate}
\end{corollary}
\begin{proof}
$(1)\Rightarrow(3)$. Since $D_i(U)\subseteq U$ for every $i$, the subspace $U$ is $L$-stable, and it contains the coordinate functions $x_1,\ldots,x_n$, which generate $\K[x_1,\ldots,x_n]$. Lemma~\ref{lem:faithful-restriction}(3) therefore gives an injective homomorphism $L\longrightarrow\End_\K(U)$, whence $\dim_\K L\leq(\dim_\K U)^2<\infty$.
For $d\geq0$, set
$$
U^{\leq d}:=\Span_\K\{u_1\cdots u_q:0\leq q\leq d,\ u_1,\ldots,u_q\in U\}.
$$
Each $U^{\leq d}$ is finite-dimensional. Since $D_i(U)\subseteq U$, the Leibniz rule gives $D_i(U^{\leq d})\subseteq U^{\leq d}$ for every $0\leq i\leq r$. Since $U$ contains $1,x_1,\ldots,x_n$, every polynomial belongs to $U^{\leq d}$ for some $d$. Thus every $D_i$ is locally finite on $\K[x_1,\ldots,x_n]$.

$(3)\Rightarrow(1),(2)$. Since $L$ is finite-dimensional and generated by the locally finite derivations $D_0,\ldots,D_r$, Proposition~\ref{prop:fd-direction} shows that $L$ acts locally finitely on $\K[x_1,\ldots,x_n]$. In particular, every element of $L$ is locally finite, giving~(2). (1) follows immediately.

$(2)\Rightarrow(3)$. Since $L$ is finitely generated and contained in $\LFD(\K[x_1,\ldots,x_n])$, Theorem~\ref{thm:main} gives $\dim_\K L<\infty$.
\end{proof}

\subsection*{Acknowledgments}
The author is grateful to Mikhail Zaidenberg for his careful comments on an earlier version of this paper, which led to several improvements. The author also acknowledges support from the Alexander von Humboldt Foundation through a Bessel Research Award.

\end{document}